\documentclass[10pt]{article}%
\usepackage{amsmath}
\usepackage{amsfonts}
\usepackage{amssymb}

\usepackage[all]{xypic}
\usepackage{hyperref}
\usepackage{amsmath,amsfonts,amsthm,epsfig}

\input cyracc.def
\font \tencyr=wncyr10
\newfam\cyrfam
\font\tencyr=wncyr10
\font\sevencyr=wncyr7
\font\fivecyr=wncyr5

\textfont\cyrfam=\tencyr \scriptfont\cyrfam=\sevencyr
\scriptscriptfont\cyrfam=\fivecyr

\let\kappa\varkappa

\newcommand{\R}{\mathbb{R}}
\newcommand{\Z}{\mathbb{Z}}
\newcommand{\F}{\mathbb{F}}
\newcommand{\C}{\mathbb{C}}

\newcommand{\Nh}[2]{[\hspace{-2pt}[ #1,#2]\hspace{-2pt}]}

\newcommand{\Span}[1]{\mathrm{span\,}\{#1\}}

\newcommand{\sym}{\mathrm{sym\,}}

\newcommand{\REF}[1]{{\normalfont (\ref{#1})}}

\newcommand{\Stab}{\mathrm{Stab\,}}

\newcommand{\End}{\mathrm{End\,}}

\newcommand{\GL}{\mathrm{GL\,}}
\newcommand{\Lie}{\mathrm{Lie\,}}

\newcommand{\df}{\stackrel{\mathrm{def}}{=}}

\newcommand{\virg}[1]{``#1''}

\newcommand{\id}{\mathrm{id}}

\newcommand{\rank}{\mathrm{rank\,}}
\newcommand{\codim}{\mathrm{codim\,}}

{\theoremstyle{remark}\newtheorem{remark}{Remark}}
{\theoremstyle{remark}}

\newtheorem{theorem}{Theorem}
\newtheorem{lemma}{Lemma}
\newtheorem{corollary}{Corollary}
\newtheorem{definition}{Definition}
\newtheorem{proposition}{Proposition}
\begin{document}
\date{\today}
\title{The Bianchi Variety}
\author{G. Moreno\thanks{Dipartimento Matematica e Informatica, Universit\`{a} degli Studi di
Salerno, Via Ponte Don Melillo, 84084 Fisciano (SA), Italy. 
Levi--Civita Institute, via Colacurcio 54, 83050 Santo Stefano del Sole (AV), Italy, \texttt{http://www.levi-civita.org/}. 
 Partially supported by   P.R.I.N. national grant 2008, \virg{EX 60\%} 2008 departmental research fund, and I.I.S.F. scholarship 2008. 
 Email: \texttt{gmoreno@unisa.it}}}
%\author{\firstname{A.~M.}~\surname{Vinogradov}}
%email{vinograd@unisa.it}
%\affiliation{Dipartimento di Matematica e Informatica, Universit\`{a} degli Studi di
%Salerno, Via Ponte don Melillo, 84084 Fisciano (SA), Italy,}
%\affiliation{INFN, Gruppo Collegato di Salerno, Italy.}

\maketitle
 
%\affiliation{INFN, Gruppo Collegato di Salerno, Italy.}

\begin{abstract}

The totality $\Lie(V)$ of all Lie algebra structures on a vector space $V$ over a field  $\F$  is an algebraic  variety over $\F$ on which  the group $\GL(V)$ acts naturally. We give an explicit description of $\Lie(V)$ for $\dim V=3$ which is based on the notion of compatibility of Lie algebra structures.

%and its quotient space with respect to the action of $\GL(\F,n)$, i.e., the moduli space of Lie algebra structures on $V$. The concept of \emph{compatibility} of Lie algebra structures will be introduced. It will be shown how the description of $\frac{\Lie(\F,3)}{\GL(\F,3)}$,  and the computation of the cohomology and the deformations of a Lie algebra structure, is sensibly simplified by using the language of differential forms and the compatibility relation.   

%Finally, an algorithm for realizing any element of $\Lie(\R,3)$ as a vector field on $\R^3$, will be presented.
\par

\end{abstract}

\tableofcontents

\section*{Introduction}
  \subsubsection*{Historical remarks}
The problem of classifying 3--dimensional Lie algebras over $\R$ was firstly solved by   L. Bianchi at the end of the eigtheen century. Recently, various works concerning classifications of low--dimensional Lie algebras appeared (see, for instance,  \cite{Marmo1994} for a list of 4--dimensional Lie algebras   and \cite{Turkowski1988} for a special  list of real Lie algebras  of dimension $\leq 8$). Now Bianchi classification can be obtained in a more elegant coordinate--free manner. For instance,  in \cite{Agaoka1999} this is done on the basis of   the   invariants of   Lie structures, in \cite{Otto2005}  the co--differential graded calculus is used, in \cite{Otto2005} the outer derivations, etc. A shortcoming of the original Bianchi method, as well as of the above--cited works, is that they do not allow   a satisfactory description of deformations of 3--dimensional Lie algebras (see \cite{Otto2005, Fia2007, Marmo2001} and references therein).\par

%However, his  elementary techniques soon became inadequate, when  higher--dimensional classification problems (see, e.g., \cite{Marmo1994}, and \cite{Turkowski1988} for $\dim V\leq 8$) and new concepts like  deformations (see \cite{Otto2005} and references therein)  and contractions (see \cite{Fia2007, Marmo2001} and references therein)   began to be considered. There are recent papers where the Bianchi classification is re--obtained by using more coordinate--free methods, e.g., the invariants of a Lie structure (see \cite{Agaoka1999} and references therein),  the co--differential graded calculus (\cite{Otto2005}), the outer derivations (\cite{Otto2005}), etc. \par
 %
 
 It should be especially stressed   the recently emerged important role of Poisson geometry in various questions related with Lie algebras and, first of all, classification, representation, deformations, etc. (see \cite{Marmo1994}, \cite{MarmoVilasiVinogradov1998} and \cite{Marmo2003}). We shall exploit it throughout the paper.
 
%In particular, the Poisson geometry, whose interrelationship with Lie structures was well--known to Lie himself, was successfully exploited to recast  the original Bianchi classification in a  more general setting   (see \cite{Marmo1994} and \cite{Marmo2003}). 
  \subsubsection*{Aim of the paper}
  %
%$\F$ is assumed to be a  field of characteristic different than 2,  unless  the sentence \virg{real case} is  added, meaning $\F=\R$.  
%
Let $V$ be a vector space over field $\F$ of characteristic different from 2. All Lie algebra structures on $V$ form an algebraic variety denoted by $\Lie(V)$. We call it \virg{the Bianchi variety} if $\dim V=3$. 
The aim of this paper is to describe the Bianchi variety in a geometrically transparent manner.\par Our approach is based on the notion of   \emph{compatibility}   of Lie structures (see, for instance, \cite{MarmoVilasiVinogradov1998})  and   \emph{differential calculus over the \virg{manifold} $V^*$} in the spirit of \cite{Jet}.  First we show  that all three--dimensional unimodular  Lie algebra structures form an algebraic variety $\Lie_0(V)$  which is naturally identified with the space of symmetric bilinear forms on $V^*$. %A three-dimensional algebra structure splits canonically into unimodular In our approach   the Poisson geometry of a Lie structure is fully exploited, and   a reduced amount of  computations is required.\par 
Recall that a Lie algebra is \emph{unimodular} if operators of its adjoint representation are traceless.   Then we show that a generic Lie structure can be obtained by adding a \emph{non--unimodular \virg{charge}} to a unimodular structure. This \virg{charge} (see pag. \pageref{eqDisassemblaggioCanonico}) is a particular non--unimodular structure, which reduces the problem to a description of how such a \virg{charge} can be attached to unimodular structures.\par
The obtained description of $\Lie(V)$ allows, besides others, to see directly peculiarities of deformations of 3--dimensionale Lie structures. Also from this point of view the Bianchi classification can be seen as moduli space   
%
%
%More precisely, the linear structure of $V$ allows to determine a unique (up to scalars) volume form $\boldsymbol{x}$ on $V^*$. If $\boldsymbol{x}$ is conserved by the Hamiltonian vector fields, then the corresponding  Poisson structure    is called \emph{unimodular}. To any Lie structure one can then associate its so--called \emph{modular class}, which measures the deviance from unimodularity. It is a privilege of three--dimensional case, that the modular class is once again a Lie structure.\par
%The main idea of this paper is to classify separately the unimodular Lie structures and the modular classes, called, in the present three--dimensional framework, \emph{non--unimodular charges}. This allows to \virg{assemble}  any Lie structure  by adding   an unimodular Lie structure with a compatible non--unimodular charge. The fact that not all non--unimodular charges are compatible with a fixed unimodular structure is captured by the 2--cocycles of a Lie structure, and it will reveal a nice fibered structure in $\Lie(V)$, which in its turn simplifies the study of the  moduli space 
$\frac{\Lie(V)}{\GL(V)}$. %Notably, the   classification of non--unimodular charges is independent on the field $\F$.\par
%  All  concepts naturally associated with a Lie algebra structure, center, derived algebra, solvability,  2-cohomology, symplectic foliation, deformation, etc.,  become immediately visible in the proposed picture. When this \virg{atomistic} philosophy is pushed to its limit, we find out that the uncuttable     Lie structure is the   algebra $\mathfrak{g}$ whose only nontrivial commutation relation is $[x_1,x_2]=x_2$. All other Lie structures are sums of $\mathfrak{g}$--type structures.\par
    \subsubsection*{Notations and preliminaries}
    %Unless stated differently, $\F$ is a field of characteristic different than 2, and $V$ is a 3--dimensional $\F$--vector space.  
    We shall use the Einstein summation convention, assuming that the  index \virg{$i$} in $\frac{\partial}{\partial x_i}$  is treated as an \emph{upper} one. %Symbol \virg{$\bullet$} represents a direct sum, e.g., $S_\bullet(V)=\bigoplus S_i(V)$, but it is skipped in the case of the algebra $S_\bullet(V)$, denoted simply by  $S(V)$. 
    \par
By a  \emph{Lie structure} on a   vector space $V$ we mean a skew--symmetric $\F$--bracket $[\,\cdot\, ,\, \cdot\,]$ on $V$, which fulfills the \emph{Jacobi Identity}
	\[
	[v,[w,z]]=[[v,w],z]+[w,[v,z]]\quad\forall v,w,z\in V.
\]    
Fix  a basis $\{x_1,x_2,\ldots,x_n\}$ of $V$. This induces a basis  $\{\xi^1,\xi^2,\ldots,\xi^n\}$ of $V^*$,  a volume $n$--covector   $\boldsymbol{\xi}\df  \xi^{1}\wedge\xi^{2}\wedge\cdots\wedge\xi^{n}$, and its dual   $\boldsymbol{v}$.\par
An element  $c$ of $V\otimes_\F\bigwedge^2(V^*)$ looks as $c=c^k_{ij}x_k\otimes\xi^i\wedge\xi^j$ and defines a Lie algebra structure iff 
\begin{equation}\label{eqAVrietˆBianchiInCoordinate0}
	  c_{aj}^kc_{bc}^j+c_{cj}^kc_{ab}^j+c_{bj}^kc_{ca}^j=0,\quad a,b,c,k=1,2,\ldots,n.
\end{equation}
 This way  $\Lie(V)$ is  identified with the affine algebraic variety in   $V\otimes_\F\bigwedge^2(V^*)$ determined by equations \REF{eqAVrietˆBianchiInCoordinate0},  and a Lie structure $c$ identifies with the family of its \emph{structure constants}   $\{c_{ij}^k\}$.    Obviously, a natural action of     $\GL(V)$ on $V\otimes_\F\bigwedge^2(V^*)$ leaves $\Lie(V)$ invariant, and defines an action of $\GL(V)$ on $\Lie(V)$.\par If $\dim V=3$,  we introduce the  basis $\{\boldsymbol{\xi}^h\}_{h=1,2,3}$ of $\bigwedge^2(V^*)$, $\boldsymbol{\xi}^h\df \epsilon^h_{\ ij} \xi^{i}\wedge\xi^{j} $, where $\epsilon_{\ ij}^h$ is purely skew--symmetric symbol. Then an element  $c$ of $V\otimes_\F\bigwedge^2(V^*)$ looks as
 $
	c=c_h^kx_k\otimes\boldsymbol{\xi}^h
$, where   
     $c_{ij}^k=c_h^k\epsilon_{\ ij}^{h}$, and \REF{eqAVrietˆBianchiInCoordinate0} becomes
\begin{equation}\label{eqAVrietˆBianchiInCoordinate}
	\sum_i\epsilon_{mi}^{h}c_{i}^mc_{h}^k=0,\quad k=1,2,3.
\end{equation}

\section{Differential Calculus over algebra $S(V)$}

In this section elements of differential calculus over $V^*$ are   sketched, in the the spirit of differential calculus over commutative algebras (see \cite{Jet}). Below  $V$ stands for e a finite--dimensional $\F$--vector space, $n=\dim V$, and $S(V)=\bigoplus S_i(V)$, where $S_i(V)$ is the $i$--th symmetric power of $V$. The algebra $S(V)$ is naturally interpreted as the algebra of polynomials on       $V^*$, whose  $\F$--\emph{spectrum} identifies with $V^*$. Consequently, the necessary elements of differential calculus on the \virg{manifold} $V^*$ are interpreted as those over commutative algebra $S(V)$.\par
% Indeed, an $\F$--valued unitary algebra homomorphism of $S(V)$ is uniquely determined by its values on homogeneous monomials, i.e., elements of $V$. From this perspective, $V^*$ looks like a \virg{manifold} and $S(V)$ as an algebra of $\F$--valued functions on it. In particular, $x_1,x_2,x_3$ together define a global chart. \par
Denote by $D(V^*)$ the $S(V)$--module of \emph{derivations} of the algebra $S(V)$, which we interpret as vector fields on $V^*$.  Then, obviously, the map 
%
%The next step is to make the symbols $ \frac{ \partial}{ \partial x_i}$ meaningful. By analogy with   smooth manifolds, $ \frac{ \partial}{ \partial x_i}$ should be  an $\F$--linear endomorphism of $S(V)$ taking the value $\delta^i_j$ on $x_j$, and fulfilling the Leibniz rule. Notice   that  the element $1\otimes\xi^i$ of $S(V)\otimes_\F V^*$   acts  on $V$, the 0--degree part of $S(V)$, which generates the whole algebra $S(V)$.  By using the Leibniz rule, such an action can be extended to $S(V)$. The resulting endomorphism acts precisely as  $ \frac{ \partial}{ \partial x_i}$ is expected to act. More precisely, the following correspondences
\begin{eqnarray}
D(V^*) &\longrightarrow & S(V)\otimes_\F V^* \label{eqIdCV}\\
X &\leftrightarrow & X|_{V}\nonumber\\
X_\theta=a_{i_1,\dots,i_n,i} x_1^{i_1}\cdots x_n^{i_n} \frac{ \partial}{ \partial x_i}&\leftrightarrow& a_{i_1,\ldots,i_n,i} x_1^{i_1}\cdots x_n^{i_n}\otimes\xi^i = \theta\nonumber
\end{eqnarray}
where $ \frac{ \partial}{ \partial x_i}(v_1v_2\cdots v_m)\df\sum_{i=1}^n \xi^i(v) v_1\cdots v_{i-1}v_i\cdots v_m$, is a $S(V)$--module isomorphism. Put
\begin{equation}
D_\bullet(V^*)\df\bigoplus_i D_i(V^*),
\end{equation}
where $D_i(V^*)$ is the $S(V)$--module of \emph{skew--symmetric multi--$i$--derivations}  of the algebra $S(V)$, which we interpret as $i$--vector fields on $V^*$. Then a similar isomorphism between $D_\bullet(V^*)$ and $ S(V)\otimes_\F \bigwedge^\bullet V^*$ holds. In particular,  
$ c=c_{{ij}}^kx_k\otimes\xi^i\wedge\xi^j $ 
 corresponds to the   {bi--vector field} 
\begin{equation}\label{eqPC}
{P^c}\df c_{ij}^kx_k \frac{ \partial}{ \partial x_{i}}\wedge\frac{ \partial}{ \partial x_{j}}.
\end{equation}
If $n=3$ and $c=c_h^kx_k\otimes\boldsymbol{\xi}^h$, \REF{eqPC} reads
\begin{equation}\label{eqPC3}
{P^c}\df c_h^k\epsilon_{ij}^hx_k \frac{ \partial}{ \partial x_{i}}\wedge\frac{ \partial}{ \partial x_{j}}.
\end{equation}
The algebra $S(V)\otimes_\F \bigwedge^\bullet(V^*)$ is $\Z_2$--graded. For example, linear vector fields are exactly elements of bidegree $(1,1)$. We emphasize that accordingly to \REF{eqIdCV} linear vector fields correspond to endomorphisms of vector space $V$,
  \begin{equation}\label{eqLVF}{X_\varphi}\df \varphi_i^jx_j\frac{ \partial}{ \partial x_i},\end{equation}    
where, by definition, $X_\varphi(v)=\varphi(v)$, $v\in V$.\par The \emph{Liouville vector field} on $V^*$
\begin{equation*}
X_\id=x^i\frac{\partial}{\partial x^i}
\end{equation*}
plays a special role, and is  denoted by $\Delta$. A bi--vector is called \emph{linear} when its bidegree is $(1,2)$, \emph{quadratic} if it is $(2,2)$, etc.   These definitions extend straightforwardly to all tensor fields over $V^*$.\par
 Similarly,  
\begin{equation}
\Lambda^\bullet(V^*)\df\bigoplus_i \Lambda^i(V^*),
\end{equation}
is the $S(V)$--module of \emph{polynomial differential forms} on $V^*$. Here the $S(V)$--module $\Lambda^i(V^*)$ of $i$--th order differential forms on $V^*$ is identified with the $i$--th skew--symmetric power $\bigwedge^i(S(V)\otimes_\F V)$ of the $S(V)$--module $S(V)\otimes_\F V$, which coincides with $S(V)\otimes_\F\bigwedge^\bullet(V)$.  In particular,   this   identification for $i=1$ looks as
\begin{eqnarray}
\Lambda^1(V^*) &\longrightarrow & S(V)\otimes_\F V\label{eqIdentificazioneForme}\\
\omega_q=a^i_{i_1,\ldots,i_n} x_1^{i_1}\cdots x_n^{i_n}dx_i&\leftrightarrow& a^i_{i_1,\dots,i_n} x_1^{i_1}\cdots x_n^{i_n}\otimes x_i = q.\nonumber
\end{eqnarray}
%of $\Lambda^1(V^*) $ with the $S(V)$--module $S(V)\otimes_\F V$, and a similar isomorphism between $\Lambda^\bullet(V)$ and $ S(V)\otimes_\F \bigwedge^\bullet V^*$ holds.\par
%
In view of the above isomorphisms, natural operations with multi--vector fields and differential forms, such as insertion, Lie derivative, Schouten bracket, etc.,  are easily reproduced in $S(V)\otimes_\F\bigwedge^\bullet(V^*)$ and $S(V)\otimes_\F\bigwedge^\bullet(V)$. These algebras are naturally \emph{bi--graded} ($\Z^2$--graded). The total degree of an element of bidegree $(p,q)$ is $p+q$. Obviously, 
%
% 
% 
% Dually to multi--vector fields, the $S(V)$--module 
%$S(V)\otimes_\F \bigwedge(V)$ identifies with the module of \Emph{differential forms} on $V^*$.
%For example,   a  {quadratic form} ${q}=q^{ij}x_i\otimes x_j$ on $V^*$ corresponds to the     {linear differential 1--form} 
%\[
%{\omega_q}\df q^{ij}x_i  dx_j.
%\]    
%All the above constructed tensor fields come equipped with those natural operations that are familiar in the differential--geometric setting, like de Rham differentiation, insertion, Lie derivative, wedge product,     {Schouten commutator} $\Nh{\cdot}{\cdot}$, etc. For example, a bi--vector $P$ is called \emph{integrable} if $\Nh{P}{P}=0$. A 
 a tensor field   $T$ is    {homogeneous of total degree $k$} iff $L_\Delta(T)=kT$. 
% This is an alternative method to capture the linearity of a tensor field. For instance,    a bi--vector $P$ is linear if and only if it is of degree 2. Linear differential forms, i.e., the   elements of   
%$V\otimes_\F  V$, can be  characterized as the homogeneous differential forms  of degree 2.\par
 The Schouten bracket is denoted by $\Nh{\cdot}{\cdot}$.
%
  
% 

%It is interesting to notice that the symmetric properties of quadratic forms   correspond to differential properties of the corresponding to them linear differential forms.
Elements of $S(V)\otimes_\F V^*\subset S(V)\otimes_\F\bigwedge^\bullet(V^*)$ (resp.,  $S(V)\otimes_\F V\subset S(V)\otimes_\F\bigwedge^\bullet(V)$) will be called \emph{linear}. A linear 1--form ${\omega_q}\df q^{ij}x_i  dx_j$,  is closed iff the matrix  $q=\| q^{ij}\|$ is  {symmetric}. Also, observe that if $n=3$ and   $q$ is  {skew--symmetric} then $\omega_q\wedge d\omega_q$ is  {zero}.\par
A bivector $P\in S(V)\otimes_\F\bigwedge^2(V^*)$ is called \emph{Poisson} if $\Nh{P}{P}=0$. The following fundamental correspondence, for the first time established by  S. Lie, is the starting point of the paper.
\begin{proposition}\label{propLiePoisson}
There is a one--to--one correspondence between Lie algebra structures on $V$ and linear Poisson bivectors on $V^*$. Namely,
\begin{equation}\label{eqPCcomepiacealprofessore}
c\equiv \{c_{ij}^k\}\leftrightarrow P^c=c_{ij}^k x_k \frac{\partial}{\partial x_i}\wedge \frac{\partial}{\partial x_j}.
\end{equation}
\end{proposition}
 $P^c$ given by   \REF{eqPC}  is called the \emph{Poisson bi--vector} associated with $c$, and the corresponding to it bracket is referred to as the \emph{Lie--Poission bracket} on $S(V)$ (see \cite{Marmo1994}).
\par 
 
Recall (see\cite{CabrasVinogradov1992}) that the map
\begin{equation}\label{eqDiPi}
d_P\df\Nh{P}{\,\cdot\,}:D_\bullet(V^*)\longmapsto D_\bullet(V^*),\quad P\in D_2(V^*)
\end{equation}
is a differential in $D(V^*)$, i.e., $d_P^2=0$, iff $P$ is a Poisson bivector. Moreover we have (see \cite{CabrasVinogradov1992})
\begin{proposition}
There exists an unique homomorphism $\Gamma_P:D_\bullet(V^*)\longrightarrow\Lambda^\bullet(V^*)$ of $S(V)$--algebras which is a cochain map from $(D_\bullet(V^*),d_P)$ to $(\Lambda^\bullet(V^*),d)$.
\end{proposition}
1--cocycles (resp., 1--coboundaries) of $d_P$ are called \emph{canonical} (resp., \emph{Hamiltonian}) vector fields on $V^*$ (with respect to the Poisson structure $P$ on $V^*$). The Hamiltonian vector field corresponding to the Hamiltonian function $f\in S(V)$ will be denoted by $P_f$, i.e., $P_f=d_P(f)$. It is easy to see that $P_f=-i_{df}(P)$ (the contraction of $df$ and $P$).\par
When $P=P_c$, the corresponding to the Hamiltonian vector fields foliation is referred to as the \emph{symplectic foliation} determined by $c$. \par

%If $P$ is a Poisson bi--vector, then $d_P\df\Nh{P}{\,\cdot\,}$ is a $D(V^*)$--valued derivation of $S(V)$. By the universal property of the de Rham differential, a unique $S(V)$--modules homomorphism $\Gamma_P$ exist, making the next diagram commutative.
%\begin{equation}
%\xymatrix{S(V) \ar[r]^{d_P} \ar[d]^{\id}& D(V^*) \ar[r] ^{d_P}\ar[d]^{\Gamma_P}&D_2(V^*) \ar[r]^{d_P}\ar[d]^{\Gamma_P} &\cdots\\
% S(V) \ar[r]^d & \Lambda^1(V^*) \ar[r] ^d&\Lambda^2(V^*) \ar[r] ^d&\cdots
%}
%\end{equation}
%Notice that $d_P$ and $\Gamma_P$ have been tacitly extended as a differential, the \emph{Hamiltonian differential},   and as an algebra homomorphism, the \emph{Hamiltonian map}, respectively (see \cite{CabrasVinogradov1992}, sec. 4, for more details). Notice that a bi--vector field $P$ is integrable if and only if $d_P^2=0$, i.e., if and only if it is a Poisson one.\par
%1-cocycles of $d_P$ are called \emph{canonical} vector fields. 1--coboundaries of $d_P$ are the \emph{Poisson}, or \emph{Hamiltonian} ones. We usually write $P_f$ for $d_P(f)$. By definition, $P_f=-i_{df}(P)$.
%           $f$ is the \emph{generating function} of $P_f$. When $P=P_c$, the corresponding to the Hamiltonian vector fields foliation is referred to as the \emph{symplectic foliation} determined by $c$. \par
          
From now on we shall assume that $\dim V=3$. 
The volume form $\boldsymbol{v}=dx^1\wedge dx^2\wedge dx^3$  determines a standard duality between  {$i$--vector fields} and $(3-i)$--differential forms. The linear bi--vector $P^c$ defined by \REF{eqPC3} is dual to the     {linear 1--form} \[\alpha_c=\sum_hc_{{h}}^kx_kdx_h,\]
i.e., $P^c(f,g)\boldsymbol{v}=df\wedge dg \wedge\alpha_c$, $f,g\in S(V)$. 

We have (see\cite{MarmoVilasiVinogradov1998})
\begin{lemma}\label{lmLemmaIntegrabilitˆ}
$P\in D_2(V^*)$ is Poisson iff $\alpha\wedge d\alpha=0$ for the dual to $P$ 1--form $\alpha$.
\end{lemma}
Denote by $q_c$ the bilinear form on $V^*$ corresponding to $\alpha_c$ in \REF{eqIdentificazioneForme}.
\begin{corollary}
If   $q_c$ is either symmetric, or skew--symmetric, then $c$ is a a Lie structure on $V$.
\end{corollary}
\begin{proof}
Directly from Lemma \ref{lmLemmaIntegrabilitˆ} and Proposition \ref{propLiePoisson}.
\end{proof}

%It is straightforward to verify that 
%$P^c$ is integrable if and only if $\alpha_c\wedge d\alpha_c=0.$ (see, e.g.,\cite{MarmoVilasiVinogradov1998}).  
% Then, in view of  Lemma \ref{lemSimmetria}, if the bilinear form $q_c$ corresponding to $\alpha_c$ is  {either symmetric or skew--symmetric}, then $P^c$ is integrable. Finally, by using the equation \REF{eqAVrietˆBianchiInCoordinate}, one can recognize that  $P^c$ is  {integrable} if and only if $c$ is a  {Lie structure} on $V$.  Consequently, if  {$q_c$} is either symmetric of skew--symmetric, then  {$c$} is a Lie structure. 

Hence $\Lie (V)$ can   be identified with \emph{a subset in the space of linear differential 1--forms}.    As such, it contains  the subspace $\Lie_0(V)$   of    differential forms which correspond to symmetric bilinear forms on $V$ in  \REF{eqIdentificazioneForme}, and the subspace $N$ of those which correspond to   {skew--symmetric} differential forms. Recall that a structure $c$ is unimodular if and only if $\alpha_c$ is symmetric (see\cite{MarmoVilasiVinogradov1998}). Accordingly, elements of $\Lie_0(V)$ (resp., $N$) are called unimodular (resp., \emph{purely non--unimodular}).\par

Since a bilinear form splits into the sum of a symmetric and a skew--symmetric part, a Lie structure $c$ on $V$ can be \virg{disassembled} into the sum of an unimodular component  with  a purely non--unimodular one,  
\begin{equation}\label{eqDisassemblaggioCanonico}
\alpha_c=dF+\alpha,\quad dF\in\Lie_0(V),\alpha\in N.
\end{equation}
In terms of Lie structures, \REF{eqDisassemblaggioCanonico} reads $c=c_F+c_\alpha$, where $c_F$ (resp., $c_\alpha$) is the Lie structure corresponding to $dF$ (resp., $\alpha$), and in terms of brackets,
\begin{equation*} 
[v,w]=[v,w]_0+[v,w]_1, \quad v,w\in V
\end{equation*}
where $[\,\cdot\, , \,\cdot\,]$ (resp., $[\,\cdot\, , \,\cdot\,]_0$, $[\,\cdot\, , \,\cdot\,]_1$) is the Lie bracket on $V$ corresponding to $c$ (resp., $c_F$, $c_\alpha$).\par
%Conversely, when trying to \virg{assemble} a Lie structure out of such basic components,  one finds out that such a process is not always successful. This motivates the next  definition, whose first appearance can be traced back to \cite{MarmoVilasiVinogradov1998}.
%From the opposite point of view,   $c_F$   can be \virg{assembled} with a purely non--unimodular one $c_\alpha$, only if $c_F$ and $c_\alpha$ are \emph{compatible} in the following sense.
Recall the following 
\begin{definition}
Elements $c_1, c_2\in \Lie(V)$ are said to be \emph{compatible} if $c_1+c_2\in \Lie(V)$.
\end{definition}
So, the unimodular part $c_F$ of $c$ and its purely non--unimodular part $c_\alpha$ are compatible.\par
Disassembling \REF{eqDisassemblaggioCanonico} can also be read as $\alpha_c=\pi_0(\alpha_c)+\alpha$, where
\begin{equation}\label{eqPiZero}
V\otimes_\F{\textstyle\bigwedge^2(V^*)}\stackrel{ \pi_0}{\longrightarrow} \Lie_0(V)
\end{equation}
  is the canonical projection   of bilinear forms onto symmetric ones. %Such a point of view will be used in the sequel.

\begin{remark}
%In higher dimensions, the disassembling is still possible, but not canonical. 
The possibility to identify Lie structures as a bilinear forms is a peculiarity of the three--dimensional case only.%, in which unimodular structures corresponds to symmetric forms. 
\end{remark}

The compatibility condition of two Lie structures are, obviously, expressed in terms of their unimodular and purely non--unimodular part as follows.
\begin{lemma}
Lie structures $c_{dF+\alpha}$ and $c_{dG+\beta}$    are compatible if and only if
\begin{equation}
\Nh{c_F}{c_\beta}+\Nh{c_G}{c_\alpha}=0, 
\end{equation}
or, equivalently,
\begin{equation}
dF\wedge\beta+dG\wedge\alpha=0.
\end{equation}
\end{lemma}
The following fact is obvious as well.
\begin{lemma}\label{lemLemmaStupidoMaUtile}
Let $P$ and $Q$ be commuting bi--vectors, and $v,w\in V$, Then $\Nh{vP}{wQ}=vP(w)\wedge Q-wQ(v)\wedge P$.
\end{lemma}

\section{Finite and infinitesimal $\GL(V)$--actions}

Fix an   {automorphism $\psi$}$\in\GL(V)$. The adjoint to $\psi$ map is a   {diffeomorphism of $V^*$}, which we still denote by $\psi$. Indeed, $\psi$, the diffeomorphism,   corresponds (in the sense of \cite{Jet})  to the  algebra automorphism of $S(V)$ whose restriction to $V$ coincides with $\psi$, the automorphism.\par
Then the action of $\psi$ is naturally prolonged to differential forms and multi--vector fields on $V^*$, and, in view of isomorphisms \REF{eqIdCV} and \REF{eqIdentificazioneForme}, to the algebras $S(V)\otimes_\F\bigwedge^\bullet (V^*)$ and $S(V)\otimes_\F\bigwedge^\bullet (V)$, respectively. We keep the same symbol $\psi$ for the prolonged automorphism, except for differential forms, when   the pull-back $\psi^*$ is used.\par     An easy consequence of Lemma \ref{lmLemmaIntegrabilitˆ}   is that the action of $\GL(V)$ on linear differential 1--forms restricts to   $\Lie(V)$. In terms of Lie brackets this action reads
	\[
	[v,w]'\df \psi^{-1}([\psi(v),\psi(w)]),\quad v,w\in V,
\]
where  $[\,\cdot\, , \,\cdot\,]$   (resp., $[\,\cdot\, , \,\cdot\,]'$) corresponds to $\alpha_c$, (resp., $\psi^*(\alpha_c)$). It is straightforward to verify that $P^{\psi(c)}=\psi(P^c)$. \par 
%Since in Poicar\'e duality we tacitly identified three--vectors with scalars, the action of $\psi$ on bracket coincides  with the action of $\psi$ on differential forms, up to a factor $\det\psi$,
\begin{remark}\label{remAzione}
The identification $\alpha_c\leftrightarrow P^c$ of linear 1--forms with linear bi--vector does not commute with   actions of $\GL(V)$ on them. Namely, we have
\begin{equation*}
%\psi^*(q_c)\cdot\det\psi=q_{\psi(c)},\quad \psi\in\GL(V).
\psi^*(\alpha_c)\cdot\det\psi=\alpha_{\psi(c)},\quad \psi\in\GL(V).
\end{equation*}
\end{remark}

%To perform our analysis of the moduli space $\frac{\Lie(V)}{\GL(V)}$, we shall consider the stabilizer of a Lie structure $c$, i.e., the subgroup
Denote by
	$
	\Stab(c)\df \{\psi\in\GL(V)\ |\  {\psi(c)=c}\}\subseteq\GL(V)
$
  the \emph{stabilizer} of $c$.\par
An  {endomorphism $\varphi$}$\in\End(V)$, i.e., a   {linear vector field} on $V^*$ (see    \REF{eqLVF}), can be interpreted as an \emph{infinitesimal automorphism} and, as such, it acts on tensor fields on $V^*$ by Lie derivation.     On the other hand, the  differential $d_{P^c}$  (see   \REF{eqDiPi})
%operator on    {linear} multi--vector fields (see Remark \ref{rem1}) 
%	\[
%V\otimes_\F \bigwedge(V^*)\stackrel{ d_{P^c} }{\longrightarrow} V\otimes_\F \bigwedge(V^*)
%\]
acts on $\varphi$ and produces $ d_{P^c} (\varphi)$. It is easy to verify that  
$
L_{X_\varphi}(\alpha_c)=\alpha_{d_{P^c}(\varphi)}.
$ 
The infinitesimal counterpart of the stabilizer is the \emph{symmetry    Lie sub--algebra}
	$$
	\sym(c)\df \{\varphi\in\End(V)\ |\ L_{X_\varphi}(c)=0\}\subseteq\End(V).
$$

\begin{remark}
Notice that  $ d_{P^c} (\varphi)$ is a linear bi--vector field on $V^*$, but not necessarily a Poisson one.
\end{remark}

We conclude this section by collecting basic facts about the cohomology of Lie structures (see \cite{G1} for more details), which will be used to describe the orbits of the Bianchi variety.\par
  A linear bi--vector field $P$ such that  {$ d_{P^c} (P)=0$} is called a \emph{2--cocyle of $c$}. These cocycles form  a subspace $Z^2(c)$ in $\End(V)$.    
  A linear bi--vector field $P$ such that  {$P= d_{P^c} (\varphi)$}, for some endomorphism $\varphi$, is called a \emph{2--coboundary of $c$}. The totality of 2--coboundaries is a  subspace of $Z^2(c)$ denoted by $B^2(c)$. The quotient space $H^2(c)\df \frac{Z^2(c)}{B^2(c)}$ is called the \emph{2--cohomology of $c$}.\par

 Intuitively, the tangent space at $c$ to 
    {$\Lie(V)$} may be taught as  the affine subspace parallel to \emph{$Z^2(c)$} and passing through $c$. Similarly, the tangent space at $c$ to     {$\GL(V)\cdot c$} may be viewed as the affine subspace parallel to \emph{$B^2(c)$} and passing through $c$.
  So, in \virg{smooth} points of $\Lie(V)$, we can interpret   $\dim Z^2(c)$ as the  {dimension of $\Lie(V)$ at $c$},       $\dim B^2(c)$ as  {the dimension of the orbit of $c$},   and  the difference $\dim Z^2(c) - \dim B^2(c)=\dim H^2(c)$ as its  {co--dimension}.

\section{The Canonical Disassembling of a 3--Dimensional Lie Structure}\label{secSmontaggioCanonico}

Firstly observe that $\GL(V)$  preserves the fibers of the projection $\pi_0$ of $V\otimes_\F V$ over $ \Lie_0(V)$.

%Since $\Lie_0(V)$ and $N$ are  $\GL(V)$--submodules,   the action of $\GL(V)$  {preserves the fibers of the projection  $\pi_U:V\otimes_\F\bigwedge^2(V^*)\longrightarrow U$}. [SPOSTARE PRIMA LA DEFINIZIONE DI QUESTA PROIEZIONE]

%\begin{definition}
%The sub--bundle 
%	\[
%	\zeta\df \left.\left(\pi_U\right)\right|_{\Lie(V)}
%\]
%is called the bundle of \emph{non--unimodular 2--cocycles}.
%\end{definition}
%\begin{definition}
Put $Z_N^2(dF)\df Z^2(c_F)\cap N$. 
%is called the space of \emph{non--unimodular  2--cocycles of $dF$}.
%   \end{definition}
The following assertion is a direct consequence of the above definitions.
   \begin{proposition}\label{corCompatibilitˆUN}
In the above notation the following conditions are equivalent:
\begin{itemize}
\item $dF+\alpha$ corresponds to a Lie structure,
\item  $c_F$ and $c_\alpha$ are compatible,
\item   $\Nh{c_F}{c_\alpha}=0$,
\item  $dF\wedge\alpha=0$,
\item $\alpha\in Z_N^2(dF)$.
\end{itemize}
\end{proposition}
An easy consequence of Proposition \ref{corCompatibilitˆUN} is the following
\begin{lemma}
$Z_N^2(dF)=\zeta^{-1}(c_F)$, with $\zeta\df \left.\left(\pi_0\right)\right|_{\Lie(V)}$.
\end{lemma}
 Note that the map $\zeta$ is not of constant--rank. Namely, the dimension of $\zeta^{-1}(c_F)$ depends on the rank of the polynomial $F$.
% 
% 
%  The next step is to show that the number of independent purely non--unimodular charges which can be attached to a given unimodular structure $dF$ equals
It should be stressed that $\zeta^{-1}(c_F)$ is naturally interpreted as a variety of purely non--unimodular  structures compatible with $dF$. We shall show that its dimension equals   $3-\rank (dF)$.\par
To this end, we compute the Schouten brackets between the basis elements $\{x_idx_j\}_{i,j=1,2,3}$ of $\Lie_0(V)$ and the purely non--unimodular Lie structures  $\alpha_i\df \varepsilon_i^{\ i_1i_2}x_{i_1}dx_{i_2}$, i.e., the basis elements of $N$.\par We usually write ${\textstyle\frac{1}{2}}d(x_i^2)$ instead of $x_idx_i$, $i=1,2,3$. 
\begin{proposition}\label{propProposizioneCommutatoriMisti}
\begin{equation*}
\Nh{c_{{\textstyle\frac{1}{2}}d(x_i^2)}}{ c_{\alpha_j}}=\left\{\begin{array}{ll}0 & \textrm{ if } j\neq i , \\2x_j\boldsymbol{\xi} & \textrm{ otherwise;}  \end{array}\right.\quad \Nh{c_{{\textstyle\frac{}{}}dx_{i_1}x_{i_2}}}{ c_{\alpha_j}}=\left\{\begin{array}{ll} 0 & \textrm{ if } j \neq i_1,i_2, \\ 2 x_{i_1}\boldsymbol{\xi}  & \textrm{ if }   j=i_2,  \\ 2 x_{i_2}\boldsymbol{\xi}  & \textrm{ if }   j=i_1. \end{array}\right. 
\end{equation*}
\end{proposition}
\begin{proof} 
From $d\alpha_j=  \varepsilon_j^{\ i_1i_2}  dx_{i_1}\wedge dx_{i_2}$ it follows that $dx_i^2\wedge\alpha_j=0$ when $j\neq i$ and $dx_{i_1}x_{i_2}\wedge\alpha_j=0$ when $ j \neq i_1,i_2$. Then,  in view of Corollary \ref{corCompatibilitˆUN}, this gives the result for $i\neq j$ and for $i\neq i_1,i_2$.  \par
Next, by using Lemma \ref{lemLemmaStupidoMaUtile} we have
\begin{align*}
\Nh{c_{x_1dx_1}}{c_{x_2dx_3-x_3dx_2}} &= \Nh{x_1\xi^2\wedge\xi^3}{x_2\xi^1\wedge\xi^2-x_3\xi^3\wedge\xi^1}\\
&={x_1\xi^3\wedge\xi^1\wedge\xi^2+x_1\xi^2\wedge\xi^3\wedge\xi^1}=2x_1\boldsymbol{\xi}.
\end{align*}
%With the same technique, compute $\Nh{dx_1x_2}{\alpha_1}$:
%\begin{align*}
%\Nh{x_1dx_2+x_2dx_1}{x_2dx_3-x_3dx_2} &=\Nh{x_1dx_2}{x_2dx_3}-\Nh{x_1dx_2}{x_3dx_2}\\&+\Nh{x_2dx_1}{x_2dx_3}-\Nh{x_2dx_1}{x_3dx_2}\\
%&=\Nh{x_1\xi^3\wedge\xi^1}{x_2\xi^1\wedge\xi^2}-\Nh{x_1\xi^3\wedge\xi^1}{x_3\xi^3\wedge\xi^1}\\&+\Nh{x_2\xi^2\wedge\xi^3}{x_2\xi^1\wedge\xi^2}-\Nh{x_2\xi^2\wedge\xi^3}{x_3\xi^3\wedge\xi^1}\\
%&=-x_2\xi^2\wedge\xi^3\wedge\xi^1+x_2\xi^3\wedge\xi^1\wedge\xi^2\\
%&-x_2(-\xi^1)\wedge\xi^2\wedge\xi^3-x_2(-\xi^2)\wedge\xi^3\wedge\xi^1\\
%&=2x_2
%\end{align*}
%And $\Nh{dx_1x_2}{\alpha_2}$:
%\begin{align*}
%\Nh{x_1dx_2+x_2dx_1}{x_3dx_1-x_1dx_3} &=\Nh{x_1dx_2}{x_3dx_1}-\Nh{x_1dx_2}{x_1dx_3}\\&+\Nh{x_2dx_1}{x_3dx_1}-\Nh{x_2dx_1}{x_1dx_3}\\
%&=\Nh{x_1\xi^3\wedge\xi^1}{x_3\xi^2\wedge\xi^3}-\Nh{x_1\xi^3\wedge\xi^1}{x_1\xi^1\wedge\xi^2}\\&+\Nh{x_2\xi^2\wedge\xi^3}{x_3\xi^2\wedge\xi^3}-\Nh{x_2\xi^2\wedge\xi^3}{x_1\xi^1\wedge\xi^2}\\
%&=x_1\xi^1\wedge\xi^2\wedge\xi^3-x_1 (-\xi^3)\wedge\xi^1\wedge\xi^2 \\
%&- (-x_1\xi^2\wedge\xi^3\wedge\xi^1)-(-x_1(-\xi^1)\wedge\xi^2\wedge\xi^3)\\
%&=2x_1
%\end{align*}
Similarly one computes the remaining commutators.
%The remaining relations are obtained by cyclically permute basis elements.
\end{proof}
\begin{lemma}\label{lemCNC}
$ {\codim} Z_N^2(dF)= {\rank}(dF)$.
       \end{lemma}
\begin{proof}
Let $F={\textstyle\frac{1}{2}}(\lambda x_1^2 +\mu x_2^2+\nu x_3^2)$ and $\alpha=a\alpha_1+b\alpha_2+c\alpha_3$. Then
\[ \Nh{c_F}{c_\alpha}=(2\lambda  {a} x_1+2\mu \emph{b} x_2+2\nu  {c} x_3)\boldsymbol{\xi}
\]  is zero if and only if  {the $\F$--valued vector 
 $
(\lambda a, \mu b, \nu c)
$
 vanishes}.
\end{proof}
\begin{figure}\caption{The Bianchi variety.}\label{BV}
\epsfig{file=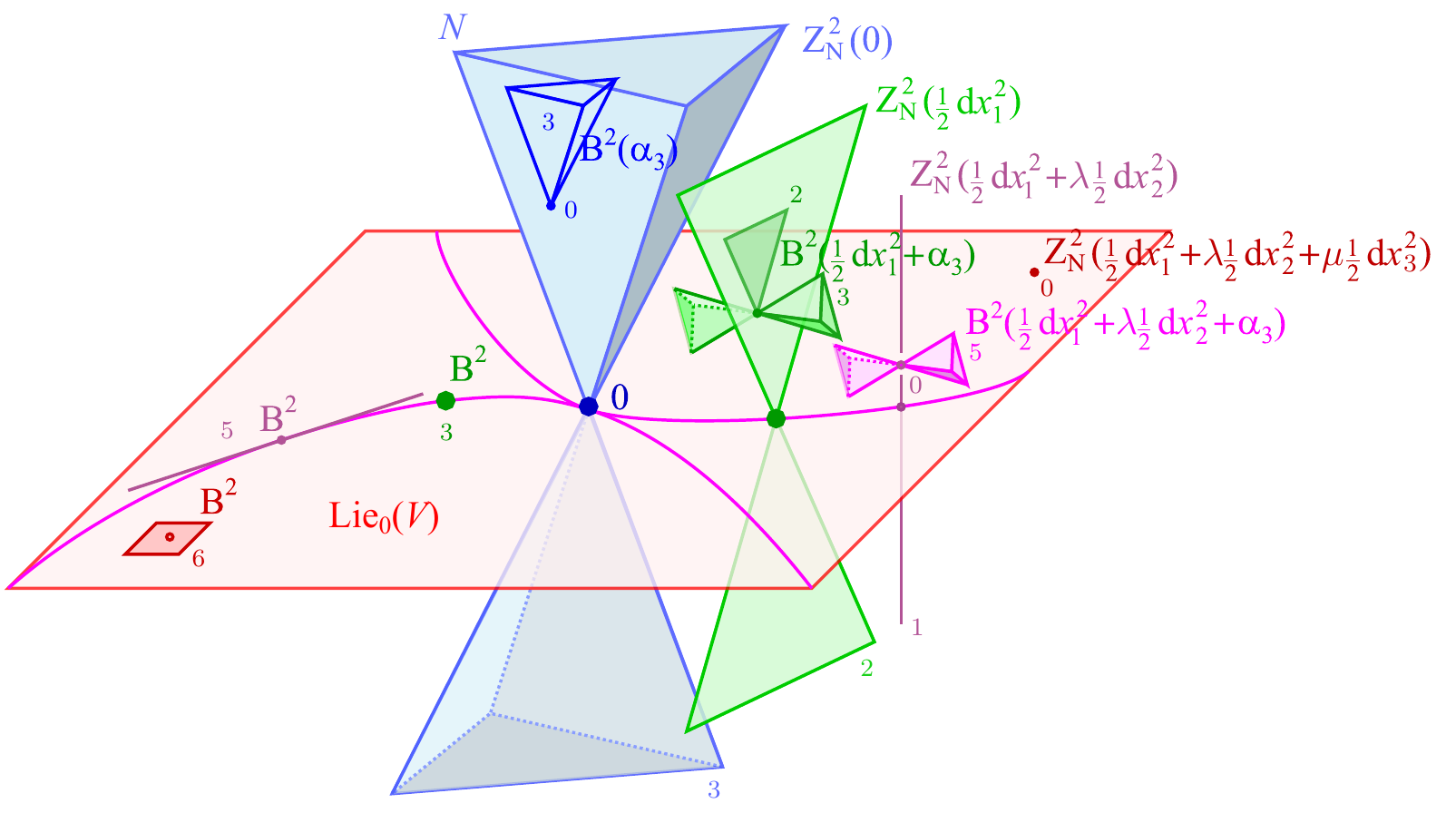, width=11.5cm}
\end{figure}
Figure \ref{BV}   visualizes   Lemma \ref{lemCNC}. The four \virg{vertical} linear spaces, crossing the \virg{horizontal} plane $\Lie_0(V)$, represent the $\zeta$--fibers attached to the rank--0 Lie structure (blue point), to a rank--1 structure (green point), to a rank--2 structure (purple poin), and to a non--degenerate structure (red point).\par
The disassembling property of   Lie structures leads  to a natural factorization of the action of $\GL(V)$ on $\Lie(V)$.  Namely,  $\GL(V)$ preserves $\zeta$.
%, we obtain a \virg{coarse} action, which just shuffles the fibers of $\zeta$. The \virg{fine} action is obtained by fixing a fiber, and acting on it by the stabilizer of its base point.  In such a perspective, 
In view of that, the study of the moduli space $\frac{\Lie(V)}{\GL(V)}$ naturally splits into two steps. The  first of them is to describe  the moduli space of the symmetric bilinear forms (which is well--known for some fields $\F$), while the second is to describe the moduli space $\frac{Z^2_N(dF)}{\Stab(dF)}$.\par
%  then the moduli space of the non--unimodular 2--cocycles (which is independent on the field $\F$).\par
To this end
%, we introduce the   \emph{stabilizer bundle}, i.e.,  the sub--bundle $\sigma:\Sigma\longrightarrow U$ of the trivial bundle $U\times\GL(V)\longrightarrow U$, where 
consider the subvariety $
\Sigma\df\{(dF,\psi)\ |\ \psi\in\Stab(dF)\}\subseteq\Lie_0(V)\times\GL(V)$ and its natural projection $\sigma:\Sigma\longmapsto\Lie_0(V)$, $(dF,\psi)\longmapsto dF$. 
Now fix an orbit \emph{$\Omega\df \GL(V)\cdot dF$} of the $\GL(V)$--action on $\Lie_0(V)$ (see Remark \ref{remAzione}). Lemma \ref{lemCNC} tells precisely  that   $\zeta|_\Omega$ is a $(3-\rank dF)$--dimensional  {vector bundle} over $\Omega$. \par
Observe that  $\sigma|_\Omega$ is a principal  group bundle over $\Omega$, acting on $\zeta|_\Omega$. 
\begin{lemma}
The quotient bundle $\frac{\zeta|_\Omega}{\sigma|_\Omega}$ is endowed with  {an absolute parallelism} and, therefore, it is  {trivial}.
\end{lemma}
\begin{proof} Take $dF, dG\in\Omega$, and choose $\varphi\in\GL(V)$ such that $dG=\varphi^*(dF)$. Define parallel displacement   $t:\left(\frac{\zeta|_\Omega}{\sigma|_\Omega}\right)^{-1}(dF)\longrightarrow\left(\frac{\zeta|_\Omega}{\sigma|_\Omega}\right)^{-1}(dG)$,
\begin{equation}\label{eqDefParTr}
t(\Stab(dF)\cdot (dF+\alpha))\df\Stab(dG)\cdot(dG+\varphi^*(\alpha)),
\end{equation}
and prove that \REF{eqDefParTr} does not depend on the choice of $\alpha$ and $\varphi$.\par
If $\alpha'$ is another choice of the non--unimodular charge of the orbit of $dF+\alpha$, then $\alpha'=\phi^*(\alpha)$, with $\phi\in\Stab(dF)$. So, $\varphi^{-1}\phi\varphi\in\Stab(dG)$ implies that $\Stab(dG)\cdot (dG+\varphi^*(\alpha))=\Stab(dG)\cdot (dG+(\varphi^{-1}\phi\varphi)^*(\varphi^*(\alpha)))=\Stab(dG)\cdot (dG+\varphi^*(\phi^*(\alpha)))=\Stab(dG)\cdot (dG+\varphi^*( \alpha'))$.\par
If $\overline{\varphi}$ is another   transformation such that $dG=\overline{\varphi}^*(dF)$, then   $\varphi^{-1}\overline{\varphi}\in\Stab(dG)$. Hence, $\Stab(dG)\cdot (dG+\varphi^*(\alpha))=\Stab(dG)\cdot (dG+(\varphi^{-1}\overline{\varphi})^*(\varphi^*(\alpha)))=\Stab(dG)\cdot (dG+\overline{\varphi}^*(\alpha))$.
 \end{proof}
Let $ c=c_F+c_\alpha$ be a Lie structure.     The orbit $GL(V)\cdot\alpha_c$ of $\alpha_c$ is precisely  {the only parallel section of $\frac{\zeta|_\Omega}{\sigma|_\Omega}$ which takes the value $\Stab(dF)\cdot\alpha$ at the point $dF$}.      
In other words, we have proved the main
\begin{theorem}\label{thTeoremaCentrale}
The orbit space $\frac{\Lie(V)}{\GL(V)}$ is fibered over the orbit space $\frac{S^2(V)}{\GL(V)}$,   the fiber at $\Omega$ being given by  {the set of parallel sections} of $\frac{\zeta|_\Omega}{\sigma|_\Omega}$.
   \end{theorem}
%\begin{remark}
%Recall that the action of $\psi\in\GL(V)$ on $q\in S^2(V)$, being inherited by the action on three--vector--valued differential forms, looks as $q\longmapsto\det\psi\cdot \psi^*(q)$. [GIA' DETTO PRIMA?]
%\end{remark}

So, we have the following algorithm for describing     {orbits of   Lie structures}:    
\begin{enumerate}
\item find the  {orbits of the action of $\GL(V)$ on $\Lie_0(V)$};    
\item find the  {parallel sections of $\frac{\zeta|_\Omega}{\sigma|_\Omega}$}, for any orbit $\Omega$ coming from the first step.
\end{enumerate}    
The evident advantage of this procedure is that the fibers of  $\zeta$ and $\sigma$ are much smaller than $\Lie(V)$ and  $\GL(V)$, respectively. Moreover, as we shall see,    the second step does not depend on the field $\F$. 
 
\begin{remark}
Even in the    case when the orbit space $\frac{S^2(V)}{\GL(V)}$ is not known,  elements of      $\Lie_0(V)$  are distinguished  by their ranks (see \cite{Lam}). Degenerate forms fill up a  {cubic hypersurface} (purple curve in  Fig. \ref{BV}), which in its turn contains a  {closed subset} of  {rank--one} forms (green points in Fig. \ref{BV}). 
\end{remark}

Let $c=c_F+c_\alpha$ be a Lie structure, and $\Omega$ the orbit of $dF$ in $\Lie_0(V)$.
\begin{lemma}\label{lemLemmaCheSembraStupidoMaNonLoEAffatto}
$\zeta|_{\GL(V)\cdot \alpha_c}$ is a   bundle over $\Omega$ with the fiber $\Stab(dF)\cdot\alpha$.
\end{lemma}
\begin{proof}
Since $\GL(V)$ acts as a bundle automorphism on $\zeta|_{\GL(V)\cdot \alpha_c}$, it suffices to compute the fiber $\zeta|_{\GL(V)\cdot \alpha_c}^{-1}(dF)$. An element $c'=c_F+c_{\alpha'}$ is in such a fiber if and only if    $dF+\alpha'\in\GL(V)\cdot\alpha_c$, i.e., $\alpha'=\psi^*(\alpha)$, with $\psi\in\Stab(dF)$.
\end{proof}
% Then we can compute  the dimension of $\GL(V)\cdot c$ by knowing the dimension of $\GL(V)\cdot dF$ and that of $\Stab(dF)\cdot\alpha$.
\begin{corollary}
$\dim\GL(V)\cdot \alpha_c=\dim(\GL(V)\cdot dF)+\dim\Stab(dF)\cdot\alpha$.
\end{corollary}
This corollary suggests a formula for computing $\dim B^2(c)$,
\begin{equation*}
\dim B^2(c)=\dim B^2(c_F)+\dim\left(\frac{\Stab (dF)}{\Stab (dF)\cap\Stab(\alpha)}\right),
\end{equation*}
whose \virg{infinitesimal version} is
\begin{equation}\label{eqFormulaBiDueStruttureMiste}
\dim\left( \frac{\End(V)}{\sym(dF+\alpha)}\right)=\dim\left( \frac{\End(V)}{\sym(dF) }\right)+\dim\left(\frac{\sym (dF)}{\sym (dF)\cap\sym(\alpha)}\right).
\end{equation}
\begin{remark}\label{remSimpaticissimo}
Notice that $\sym(dF+\alpha)=\sym(dF)\cap\sym(\alpha)$.
%, so that  formula \REF{eqFormulaBiDueStruttureMiste} above looks like a \virg{factorization} of $\frac{\End(V)}{\sym(dF+\alpha)}$, due to the disassemblement property of $c$.
\end{remark}

\section{Computations}

\subsection{Unimodular structures}
In the case $\alpha=0$ Lemma \ref{lemLemmaCheSembraStupidoMaNonLoEAffatto} says that the orbit of $\alpha_c$ coincides with $\Omega$. In view of \REF{eqFormulaBiDueStruttureMiste}, in order to find its dimension, it is sufficient to  compute $\dim[\sym(dF)]$  (Proposition  \ref{propSimmetrieSimmetriche}).
% and plug the result into   \REF{eqFormulaBiDueStruttureMiste}. To find $\dim\Lie(V)$ at points of $\Omega$, we just have to add to $\dim U$, which is six, the dimension of the  fiber of $\zeta|_\Omega$, as   Proposition  \ref{propSimpatica} will show.
\begin{proposition}\label{propSimmetrieSimmetriche}
\begin{equation*}
\dim B^2(dF)=\left\{\begin{array}{cc}6 & \textrm{\normalfont if } \rank dF=3 \\5 &\textrm{\normalfont if } \rank dF=2 \\3 &\textrm{\normalfont if } \rank dF=1\end{array}\right.
\end{equation*}
\end{proposition}
\begin{proof}
We shall  show that 
\begin{equation*}
\dim [\sym(dF)]=\left\{\begin{array}{cc}3 &\textrm{ if } \rank dF=3 \\4 &\textrm{ if } \rank dF=2 \\6 &\textrm{ if } \rank dF=1.\end{array}\right.
\end{equation*}
To this end, prove that  $\varphi\in\sym(\frac{1}{2}d(x_1^2+\lambda x_2^2+\mu x_3^2))$ if and only if 
\begin{equation}\label{eqSimmForSimmetriche1}
X_\varphi=(-\lambda a x_2-\mu  bx_3)\frac{\partial}{\partial x_1}+( ax_1+ cx_2 + ex_3)\frac{\partial}{\partial x_2}+( b x_1+f x_2 + dx_3)\frac{\partial}{\partial x_3},
\end{equation}
with the   coefficients $a,\ldots,f$ satisfying conditions
\begin{equation}\label{eqSimmForSimmetriche2}
\left\{\begin{array}{c}\lambda c=0 \\ \mu d=0 \\ \lambda e+\mu f=0.\end{array}\right.
\end{equation}
Indeed, since  $L_{X_\varphi}(\frac{1}{2}d(x_k^2))=(\varphi_i^jx_j\frac{\partial}{\partial x_i})(\frac{1}{2}d(x_k^2))=\frac{1}{2}d((\varphi_i^jx_j\frac{\partial}{\partial x_i})(x_k^2))=d(\varphi_i^jx_j\delta_k^ix_k)=\varphi_k^jdx_jx_k,$ the Lie derivative
\begin{align*}
L_{X_\varphi}({\textstyle\frac{1}{2}}d(x_1^2 +\lambda x_2^2 +\mu x_3^2))&=\varphi_1^1{\textstyle\frac{1}{2}}d(x_1^2)+\lambda\varphi_2^2{\textstyle\frac{1}{2}}d(x_2^2)+\mu\varphi_3^3{\textstyle\frac{1}{2}}d(x_3^2)\\
&+(\varphi_1^2+\lambda\varphi_2^1)d(x_1x_2)+(\mu\varphi_3^1+\varphi_1^3)d(x_3x_1)+(\lambda\varphi_2^3+\mu\varphi_3^2)d(x_2x_3)
\end{align*}
vanishes if and only if $X_\varphi$ can be put in the form \REF{eqSimmForSimmetriche1}, with coefficients satisfying \REF{eqSimmForSimmetriche2}. 
\end{proof}%[AGGIUNGERE TUTTE LE PARENTESI PER NON FAR SEMBRARE DELLE METRICHE]\par
In the left side of Figure \ref{BV} the spaces $B^2(dF)$, whose dimension was computed in Proposition \ref{propSimmetrieSimmetriche},  are drawn as tangent spaces to $\Lie_0(V)$.%, and its degeneracy loci, depending on the rank of $dF$.
% being 3, 2, and 1, respectively.
 \begin{proposition}\label{propSimpatica}
 $\dim Z^2(c_F)=9-\rank dF$.
 \end{proposition}
 \begin{proof}
 Observe that \emph{$Z^2(c_F)=\Lie_0(V)\oplus Z_N^2(dF)$} and apply Lemma \ref{lemCNC}.
 \end{proof}
Figure \ref{BV} makes evident  Proposition \ref{propSimpatica}. Indeed, $Z^2(c_F)$ is precisely the space spanned by the \virg{horizontal} subspace $\Lie_0(V)$ and the \virg{vertical} subspaces $Z_N^2(dF)$.\par
The above results concerning the orbits of unimodular structures are summarized in  the  next table for $\F=\R$.
{\small\begin{center}
	\begin{tabular}{cccccccc}
	Type &Bianchi type(s) & Lie structure(s)&$\rank dF$ & \emph{$\dim\GL(V)\cdot dF$} & $\dim Z_N^2(dF)$ & $\dim Z^2(c_F)$ & $ \dim H^2(c_F)$\\
	$A_0$ &AI & Abelian &0 & \emph{0} & 3 & 9 & 9\\
	$A_1$&AII	& Heisenberg &1 & \emph{3} & 2 & 8 & 5\\
	$A_2^-$, $A_2^+$&AVI$_0$, AVII$_0$	& $\mathfrak{e}(1,1)$, $\mathfrak{e}(2)$&	2 & \emph{5} & 1 & 7 & 2\\
	$A_3^-$, $A_3^+$&AVIII, AIX		 &	$\mathfrak{o}(2,1)$, $\mathfrak{o}(3)$ & 3& \emph{6} & 0 & 6 & 0
	\end{tabular}
\end{center}}
\begin{remark}
In this table we introduce a new notation for isomorphism classes of three--dimensional Lie algebras, hoping it will be more informative.   
%Type labels is a private notation of   this paper. 
The original Bianchi notation can be found  in  \cite{CG}.%, and its peculiar distinction of certain types can be understood by looking at the corresponding symplectic foliations.
\end{remark}
%It is interesting to notice  that the codimension of   $\Omega$ in $\Lie(V)$ is precisely the codimension of  {$\Omega$ in $U$},    {plus} the dimension of the  fiber of $\zeta|_\Omega$. 
\subsection{Non--unimodular structures}
\subsubsection{$\rank dF=0$.}
 Then   $dF=0$, $\Omega=\GL(V)\cdot dF=\{0\}$, $Z_N^2(0)=N$ and $\Stab(c_F)=\GL(V)$. In other words,     $\frac{\zeta|_\Omega}{\sigma|_\Omega}$ consists of just one fiber, which identifies with
 \begin{equation}\label{eqSpazioModuliFibraRangoZero}
\frac{N}{\GL(V)}.
\end{equation}
Independently on the field $\F$, it can be easily proved (see \cite{MarmoVilasiVinogradov1998}) the following
\begin{proposition}
The moduli space \REF{eqSpazioModuliFibraRangoZero} consists of two orbits, one of which is  0.
\end{proposition}
%  Accordingly to Lemma \ref{lemLemmaCheSembraStupidoMaNonLoEAffatto}, the orbit of $c$ is a bundle over $\{0\}$ with fiber $\GL(V)\cdot\alpha_3$. 
\begin{proposition}\label{propSimmetrieAntiSimmetriche}
$\varphi\in\sym(\alpha_3)$ if and only if  
 \begin{equation}\label{eqSimmForSkSimmetriche1}
X_\varphi=( a x_1+bx_2)\frac{\partial}{\partial x_1}+( cx_1- ax_2  )\frac{\partial}{\partial x_2}+( d x_1+e x_2 + fx_3)\frac{\partial}{\partial x_3}.
\end{equation}
\end{proposition}

\begin{proof}
It directly follows from
\begin{align*}
L_{X_\varphi}(\alpha_3)&=(\varphi_i^jx_j{\textstyle\frac{\partial}{\partial x_i}})(x_1dx_2-x_2dx_1)=\\
&=\varphi_i^jx_j\delta_1^idx_2+x_1d(\varphi_i^jx_j\delta_2^i)-\varphi_i^jx_j\delta_2^idx_1-x_2d(\varphi_i^jx_j\delta_1^i)=\\
&=\varphi_1^j(x_jdx_2-x_2dx_j)+\varphi_2^j(x_1dx_j-x_jdx_1)=-\varphi_1^3\alpha_1-\varphi_2^3\alpha_2+(\varphi_1^1+\varphi_2^2)\alpha_3.
\end{align*}

\end{proof}
  The \virg{vertical} blue  subspace in Figure \ref{BV} is $N$. The 3--dimensional space $B^2(\alpha_3)$ is shown inside $N$.\par
Notice that when $\F=\R$ or $\C$, Proposition \ref{propSimmetrieAntiSimmetriche} is sufficient to prove that the orbit of $\alpha_3$ is 3--dimensional and, therefore, it coincides  with $N\smallsetminus\{0\}$.

   \subsubsection{$\rank dF=1$.} 
Independently on the field $\F$, all rank--1 elements of $\Lie_0(V)$ belong to the same orbit    $\Omega=\GL(V)\cdot {\textstyle\frac{1}{2}} d(x_1^2)$. To compute the fiber of $\frac{\Lie(V)}{\GL(V)}$ over $\Omega$, it suffices to compute the moduli space 
\begin{equation}\label{eqSpazioModuliFibraRangoUno}
\frac{Z_N^2({\textstyle\frac{1}{2}} d(x_1^2))}{\Stab({\textstyle\frac{1}{2}} d(x_1^2))}
\end{equation}
 (see  Theorem \ref{thTeoremaCentrale}).  \par
Observe that $Z_N^2({\textstyle\frac{1}{2}} d(x_1^2))$ is the 2--dimensional vector space spanned by $\alpha_2$ and $\alpha_3$ (see the proof of Lemma \ref{lemCNC}). Fix a non--zero element  $a\alpha_2+b\alpha_3$. Then it is possible to choose an automorphism $\psi\in\GL(V)$ which preserves $x_1$ and sends $bx_2-ax_3$ to $x_2$. In other words, $\psi\in\Stab({\textstyle\frac{1}{2}} d(x_1^2))$ and $\psi^*(a\alpha_2+b\alpha_3)=\alpha_3$, thus proving the following
\begin{proposition}\label{propFibraRangoUno}
The moduli space \REF{eqSpazioModuliFibraRangoUno} consists of two orbits, one of which is  0.
\end{proposition}
      \subsubsection{$\rank dF=2$.} 
The orbits of rank--2 structures  in $\Lie_0(V)$ are $\Omega=\GL(V)\cdot dF$, $F={\textstyle\frac{1}{2}} (x_1^2+\epsilon x_2^2)$, with $\epsilon\in\F$  (see \cite{Lam}).  Recall that $Z_N^2(dF)$ is the 1--dimensional subspace spanned by $\alpha_3$ (see the proof of Lemma \ref{lemCNC}).\par
We shall show that the fiber over $\Omega$ is $\F$.
\begin{proposition}\label{propFibraRangoDue}
Let $\F$ be $\R$ (resp., $\C$). Then the moduli space 
\begin{equation}\label{eqSpazioModuliFibraRangoDue}
\frac{Z_N^2(dF)}{(\Stab dF)},\quad F={\textstyle\frac{1}{2}} (x_1^2+\epsilon x_2^2), \epsilon=\pm 1\ (\textrm{resp.}\ 1)
\end{equation}
coincides with $\langle\alpha_3\rangle$.
\end{proposition}
\begin{proof}
Notice that he stabilizer in $\Stab(dF)$ of an element $\lambda\alpha_3\in Z_N^2(dF)$ coincides with $\Stab(\lambda\alpha_3)\cap\Stab(dF)$. To prove the result, it suffices to show that $\Stab(dF)$ is contained in $\Stab(\lambda\alpha_3)$.\par
This is obvious for  $\lambda=0$. For $\lambda\neq 0$ we, first, observe that $\Stab(\lambda\alpha_3)=\Stab(\alpha_3)$. Then, it follows  from Propositions \ref{propSimmetrieSimmetriche} and \ref{propSimmetrieAntiSimmetriche}   that a symmetry of $dF$ is also a symmetry of $\alpha_3$.
\end{proof}
The proof of the above proposition is simplified by infinitesimal  arguments,   which does not work if $\F$ is different from $\R$ or $\C$. For a generic $\F$ see \cite{MarmoVilasiVinogradov1998}.

\subsubsection{Cocycles of non--unimodular Lie structures}
 
\begin{lemma}\label{lemLemmaCheRitenevoFineASeStessoMaChePoiAllaFineHoDovutoRichiamare} If $c$ is a non--unimodular Lie structures, then  $\dim Z^2(c)=6$.\end{lemma}
\begin{proof} 
Any non--unimodular Lie structure is equivalent to   $c_{{\textstyle\frac{1}{2}} (\lambda x_1^2 +\mu x_2^2)}+c_{\alpha_3}$. Let $dF=d({\textstyle\frac{1}{2}}(ax_1^2 +b x_2^2 +c x_3^2)+e x_2x_3+fx_1x_3+gx_1x_2)$ (resp., $\alpha=k\alpha_1+l\alpha_2+m\alpha_3$) be an arbitrary element of $\Lie_0(V)$ (resp., $N$). Then, independently on $\lambda$ and $\mu$, the commutator
\begin{align*}
 \Nh{c_{{\textstyle\frac{1}{2}} (\lambda x_1^2 +\mu x_2^2)}+c_{\alpha_3}}{c_F+c_\alpha}&=
  \Nh{c_{{\textstyle\frac{1}{2}} (\lambda x_1^2 +\mu x_2^2)} }{c_\alpha}\\ &+\Nh{c_{{\textstyle\frac{1}{2}}(ax_1^2 +b x_2^2 +c x_3^2)+e x_2x_3+fx_1x_3+gx_1x_2}}{c_{\alpha_3}}\\
  &=(k\lambda2x_1+l\mu2x_2+c2x_3+2ex_2+2fx_1)\boldsymbol{\xi}\\
  &=2((f+k\lambda)x_1+(e+l\mu)x_2+cx_3)\boldsymbol{\xi}
\end{align*}
vanishes if and only if the three equations $f+k\lambda=0$, $e+l\mu=0$ and $c=0$ are satisfied.
\end{proof}
The obtained results are  summarized in the following table, where $c=c_F+c_\alpha$.
\begin{center}
	\begin{tabular}{cccccc}
	Type &Bianchi type(s)  & $\rank dF$ & $\dim B^2(c)$ & $\dim Z^2(c)$  & \emph{$\dim H^2( c)$}\\
	$B_0$&V &0 & 3 & 6 & 3\\
	$B_1$&IV &	1 & 5 & 6 & \emph{1}\\
	$B_{2,\lambda}^\pm$&III, VI$_h$, VII$_h$&		2 & 5 & 6 & \emph{1}
	\end{tabular}
\end{center}

\section{Compatibility varieties}

Let $c\in\Lie(V)$.% be a Lie structure.
\begin{definition}
The affine algebraic variety $\Lie(V,c)\df\Lie(V)\cap Z ^2(c)\subseteq Z ^2(c)$ is called the \emph{compatibility variety} of $c$.%, or the \emph{relative} (to $c$) Bianchi variety.
\end{definition}
%\begin{proposition}\label{propCompatibilitˆVarie}
Obviously, $\Lie(V,c)$ can be understood as   the set of Lie  structures   which are compatible with $c$, or as   the union of all linear subspace of $\Lie(V)$ passing through $c$.
%\item   the set of tangent vectors to $\Lie(V)$ which are contained into $\Lie(V)$.
%\end{itemize}
%\end{proposition}
%\begin{proof}
%An easy consequence of [...].
%\end{proof}
%
%
So, $\Lie(V,c)$ is a conic variety.\par
% and any curve belonging to $\Lie(V,c)$ is a deformation of $c$, not necessarily linear.\par
%Any element $d\in \Lie(V,c)$ produces a linear deformation 
%\begin{equation}\label{eqDefLin}
%\gamma_d(t)\df (1-t)c+td
%\end{equation}
% of $c$. 
% On the other hand, let $\gamma$ be a linear deformation. Then, all the one--dimensional linear subspaces passing through a point of $\gamma$ are contained in $\Lie(V)$ (Remark \ref{remProiettivitˆ}). From this it follows that the image of $\gamma$ lies in a linear subspace contained in $\Lie(V)$, and as such is made by mutually compatible structures, by Proposition \ref{propCompatibilitˆVarie}.  
%\begin{lemma}
%Any linear deformation of $c\in\Lie(V)$ is of the form $\gamma_d$, $d\in \Lie(V,c)$.
%\end{lemma}
%
%
%Accordingly to some authors, two deformations of $c$ are equivalent if there exists an element of $\Stab(c)$ which sends one deformation into the other.  Then   $\frac{\Lie(V,c)}{\Stab(c)}$ may be interpreted as the moduli space of linear deformations of $c$. \par
%
%The orbit space $\frac{\Lie(V,c)}{\Stab(c)}$ is  naturally interpreted as  the moduli space of linear deformations of $c$.\par
%
%
% 
%
%
The canonical disassembling of $\Lie(V)$ and other results of Section \ref{secSmontaggioCanonico} are reproduced as well for the compatibility variety $\Lie(V,c)$, with unimodular $c$. 
%
%When $c$ is unimodular, the canonical disassembling described in Section \ref{secSmontaggioCanonico} can be straightforwardly adapted to the compatibility variety $\Lie(V,c)$. 
%
%This   simplifies description of   $\frac{\Lie(V,c)}{\Stab(c)}$ for an unimodular Lie structure $c$.
%
%\begin{remark}
%If $F={ \frac{1}{2}}(x_1^2 +   x_2^2\pm  x_3^2)$, then $\Stab(c_F)$ is   $O(3)$ (resp., $O(2,1)$) (see also Proposition \ref{propSimmetrieSimmetriche}). Similarly, for $F={ \frac{1}{2}}(x_1^2  \pm  x_2^2)$,  the group $\Stab(c_F)$ will be denoted   $O(2,0)$ or $O(1,1,0)$, respectively.  
%%A transformation which  leaves invariant the subspace  $\Span{ x_1, x_2}$   belongs to $O(2,0)$ (resp., $O(1,1,0)$) if, restricted on $\Span{x_1, x_2}$, is an element of  $O(2)$ (resp., $O(1,1)$). 
%Finally, notice that $\Stab(c_F)$, for  $F={ \frac{1}{2}}x_1^2 $,  coincides with the stabilizer of $x_1$.
%\end{remark}
%
%We do not describe the orbits of the action of $\Stab(c_F)$ on $\Lie_0(V)$, since this concerns the theory of symmetric bilinear forms (see \cite{Lam}).\par
In particular, 
$\Lie_0(V)\subseteq\Lie(V,c_F)$   for any   $F$. Consider the map $\zeta^F\df\pi_0|_{\Lie(V,c_F)}$. Then  we have 
$$
(\zeta^F)^{-1}(dG)=Z_N^2(dG)\cap Z_N^2(dF).
$$
% [SOSTITUIRE $G$ AD $F'$ OVUNQUE!]
\subsection{Computations}
In this subsection we shall describe the varieties $\Lie(V,c )$, for all types of structures $c$. Obviously, $\Lie(V,0)=\Lie(V)$, so we assume $c\neq 0$. \par
We   introduce the notation $$s^2\df \Span{{\textstyle \frac{1}{2}}d(x_1^2), {\textstyle \frac{1}{2}}d(x_2^2), {\textstyle \frac{1}{2}}d(x_1x_2)}.$$ Notice that $s^2$ identifies with the space of symmetric bilinear forms on $\Span{\xi^1,\xi^2}$.
\subsubsection{Compatibility variety of  $A_3^\pm$   structures}\label{subsubA3}

Let $c=c_F$. If  $\rank(F)=3$, then $Z_N^2(dF)=0$ and $\Lie(V,c)=\Lie_0(V)$ is a 6--dimensional vector subspace (see Fig. \ref{figComVarRank3}). 

\begin{figure}
\epsfig{file=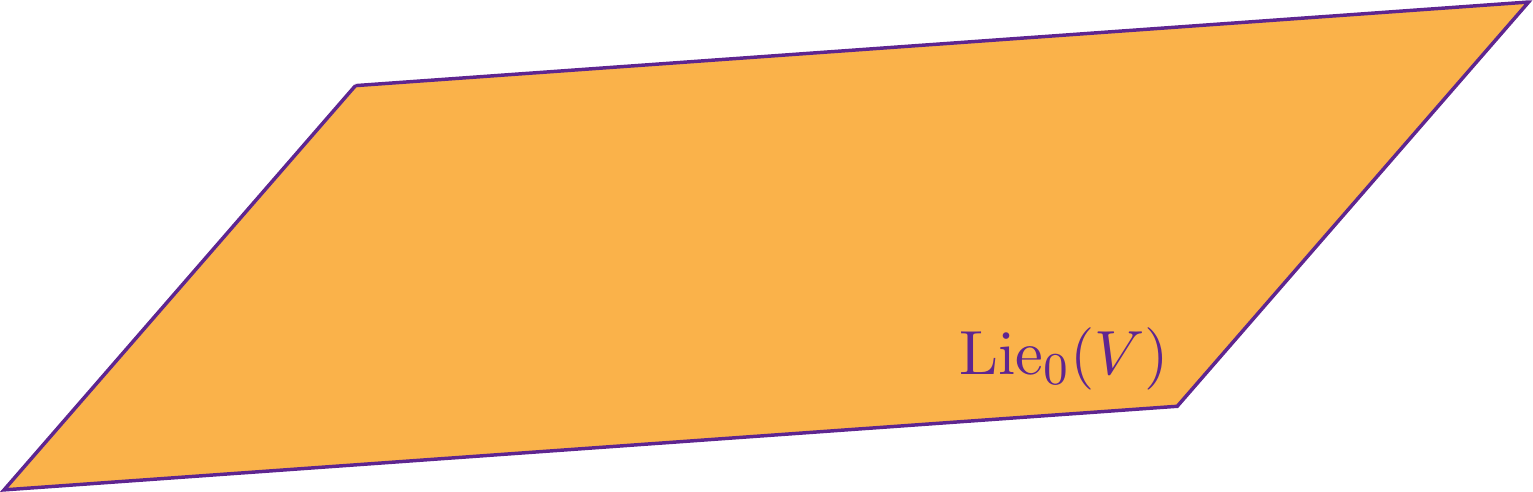,width=12cm}\caption{Only when $c$ is an unimodular  structure   of type $A_3^\pm$, $\Lie(V,c)$ is a linear space.\label{figComVarRank3}}
\end{figure}

\subsubsection{Compatibility variety of $A_2^\pm$ structures}\label{subsubA2}

Let now    $F={ \frac{1}{2}}(x_1^2 +\epsilon  x_2^2)$, $\epsilon\in\F\smallsetminus\{0\}$.
% In this case (see   [...])    $Z_N^2(dF) $ is the linear subspace generated by $\alpha_3$, so $\mathbb{P}Z_N^2(dF) $ is the single point $[\alpha_3]$.
%Finally, let $c_F$ be   of type $A_1$, e.g.,   $F=d{ \frac{1}{2}}(x_1^2  )$.
\begin{lemma}\label{lemZetaZero}
%The subspace  
$Z^2(\alpha_3)\cap\Lie_0(V)
%$ of $\Lie_0(V)$ of compatible with $\alpha_3$ unimodular Lie structures is $ 
=s^2$.
\end{lemma}
\begin{proof}
Immediately   from Proposition \ref{propProposizioneCommutatoriMisti}.
\end{proof}
\begin{proposition}
$\Lie(V,c_F)$ is the union \begin{equation}\label{eqVELA}
\Lie(V,c_F)=\Lie_0(V)\cup \Span{s^2 ,\alpha_3}
\end{equation}
of a 6--dimensional  and a  4--dimensional subspace,   intersecting along the 3--dimensional subspace $s^2$. 
\end{proposition}
\begin{proof}
Obviously, the right--hand side of \REF{eqVELA} is contained in the left one. Let  $c'=c_G+\alpha\in \Lie(V,c_F) $ with $\alpha\neq 0$.\par  Since  $c'$ is compatible with $c_F$,  $dF\wedge d\alpha_{c'}=0$. But $dF\wedge d\alpha_{c'}=dF\wedge d\alpha$, so   $dF\wedge d\alpha=0$, i.e. $\alpha\in Z_N^2(dF)$. In view of Lemma \ref{lemCNC}, $Z_N^2(dF)$ is the one--dimensional subspace generated by $\alpha_3$. Hence $\alpha=\lambda\alpha_3$, $\lambda\neq 0$.\par
This shows that   $c_G$, being compatible with $\alpha$, is compatible with   $\alpha_3$ and, by  Lemma \ref{lemZetaZero},  is a linear combination of $  \frac{1}{2}d(x_1^2),   \frac{1}{2}d(x_2^2), d(x_1x_2)$. 
\end{proof}
Figure \ref{figComVarRank2} shows that the structure of $\Lie(V,c_F)$ is quite simple. The 3--dimensional subspace $s^2$ is precisely the locus where the fibers of $\zeta^F$ are nontrivial. The restriction of $\zeta^F$ to it is a trivial bundle with fiber $\Span{\alpha_3}$. 
\begin{figure}
\epsfig{file=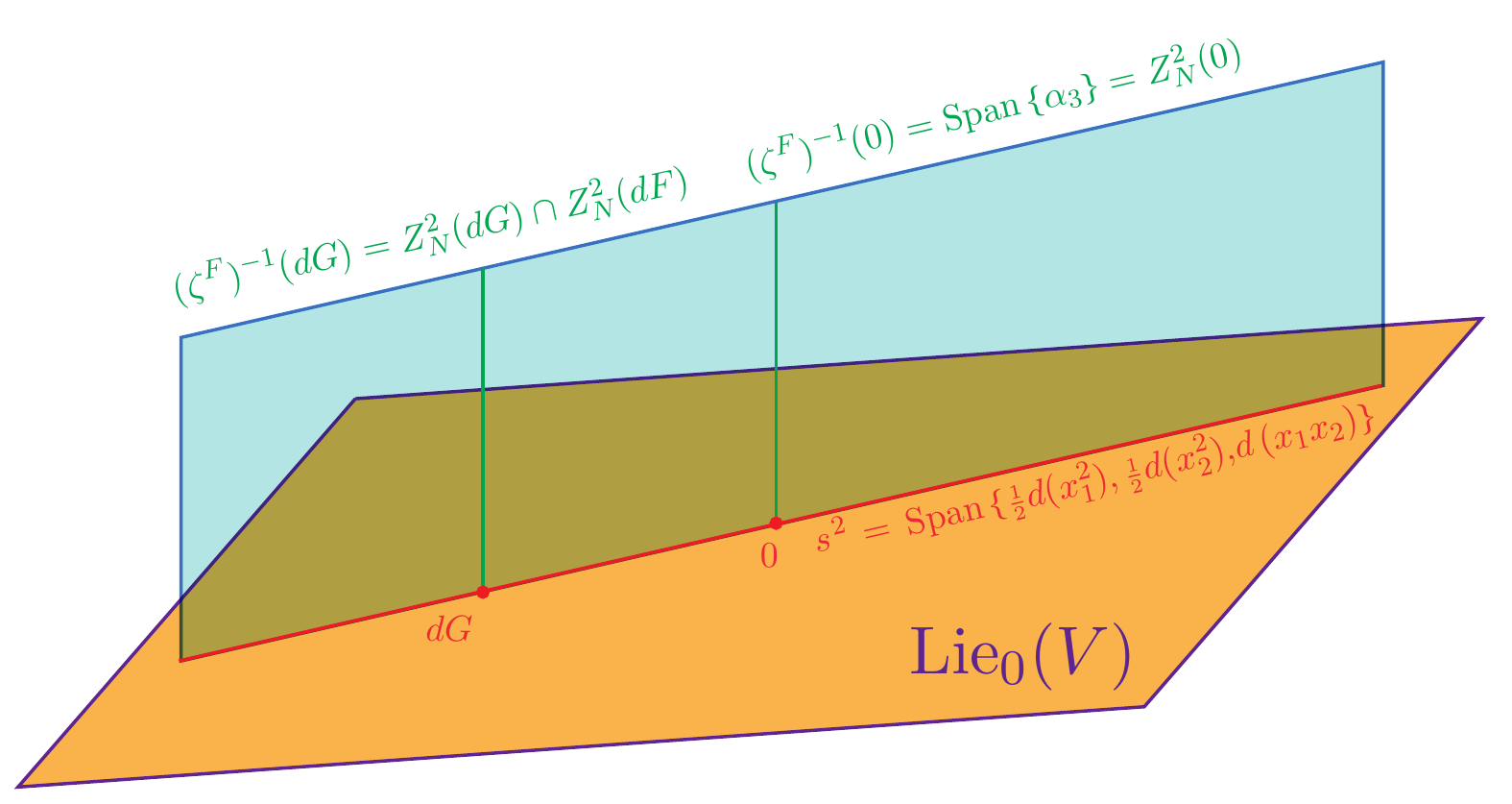,width=12cm}\caption{The compatibility variety af a  structure $c_F$ of type $A_2^\pm$.\label{figComVarRank2}}
\end{figure}

%
%Describe  $\frac{\Lie(V,c_F)}{\Stab(c_F)}$. Firstly, the subspace $\Span{ {\textstyle \frac{1}{2}}d(x_1^2),  {\textstyle \frac{1}{2}}d(x_2^2), d(x_1x_2) }$ is $\Stab(c_F)$--in\-var\-i\-ant. The restricted to it action  coincides with that of $O(2)$ or $O(1,1)$, depending on the signature of $F$, on the space of symmetric bilinear forms on a 2--dimensional $\F$--vector space, whose orbits are well--known.\par Now let $\Omega$ be   a $\Stab(c_F)$--orbit in  $\Span{ {\textstyle \frac{1}{2}}d(x_1^2),  {\textstyle \frac{1}{2}}d(x_2^2), d(x_1x_2) }$. By the same arguments of Proposition \ref{propFibraRangoDue}, it follows  that the set of parallel sections of $\frac{\zeta|^F_\Omega}{\sigma|^F_\Omega}$ is $\F$. 

%\par Analogously to [...], introduce the linear subspace $Z_0^2(\alpha)\df Z^2(\alpha)\cap\Lie_0(V)$ of $\Lie_0(V)$, for a purely non--unimodular structure $\alpha$.\par
%From [...] it follows that $Z_0^2(\alpha_3)=\langle d{ \frac{1}{2}}(x_1^2), d{ \frac{1}{2}}(x_2^2), d(x_1x_2) \rangle$. 
%\begin{lemma}
%$ Z_0^2(0)=\Lie_0(V)$, while $\dim Z_0^2(\alpha)=3$ for all non--zero purely non--unimodular Lie structures $\alpha$.
%\end{lemma}
%\begin{proof}
%Immediate consequence of [...].
%\end{proof}

%
%

\subsubsection{Compatibility variety  of $A_1$ structures}\label{subsubA1}
This case is more complicated (see Fig. \ref{figComVarRank1}). Let $F= {\textstyle \frac{1}{2}}x_1^2$.
\begin{proposition}
If $(0,0)\neq(a,b)\in\F^2$, then $\zeta^F$ is a rank--2 trivial bundle over the  line $\Span{{\textstyle \frac{1}{2}}d(x_1^2)}$  with the fiber $\Span{\alpha_2,\alpha_3}$, and over $\Span{{\textstyle \frac{1}{2}}d(x_1^2), {\textstyle \frac{1}{2}}d((bx_2-ax_3)^2)}\smallsetminus \Span{{\textstyle \frac{1}{2}}d(x_1^2)}$, $\zeta^F$ is a rank--1 trivial bundle with the fiber  $\Span{a\alpha_2+b\alpha_3}$.   Fibers of $\zeta^F$  are trivial over the rest of $\Lie_0(V)$.
\end{proposition}
\begin{proof}
As it follows from Lemma \ref{lemCNC},  $Z_N^2(dF)=\Span{\alpha_2,\alpha_3}$. Therefore, the intersection $Z_N^2(dF)\cap Z_N^2(dG)$ is 2--dimensional if and only if $Z_N^2(dF) = Z_N^2(dG)$, i.e., if $dG$ belongs to the line $\Span{{\textstyle \frac{1}{2}}d(x_1^2)}$.\par
 The intersection  $Z_N^2(dF)\cap Z_N^2(dG)$ can be of dimension 1 in the following two cases. First,  $Z_N^2(dG)$ is a 2--dimensional subspace intersecting $\Span{\alpha_2,\alpha_3}$ along a line, and, second, $Z_N^2(dG)$ is  a 1--dimensional subspace contained in $\Span{\alpha_2,\alpha_3}$.\par
In the first case, a line in $\Span{\alpha_2,\alpha_3}$ can be written as  $\Span{a\alpha_2+b\alpha_3}$, with $(0,0)\neq(a,b)$. Then $ dG= {\textstyle \frac{1}{2}}d((bx_2-ax_3)^2)$ is the only rank--1 structure  such that $Z_N^2 (dG)$ intersects  $\Span{\alpha_2,\alpha_3}$ along $\Span{a\alpha_2+b\alpha_3}$.\par
In the second case, $dG$ must be a rank--2 structure such that $Z_N^2 (dG)$ is precisely $\Span{a\alpha_2+b\alpha_3}$. Up to proportionality, this is $G={\textstyle \frac{1}{2}}(x_1^2+(bx_2-ax_3)^2)$.
\end{proof}
\begin{figure}
\epsfig{file=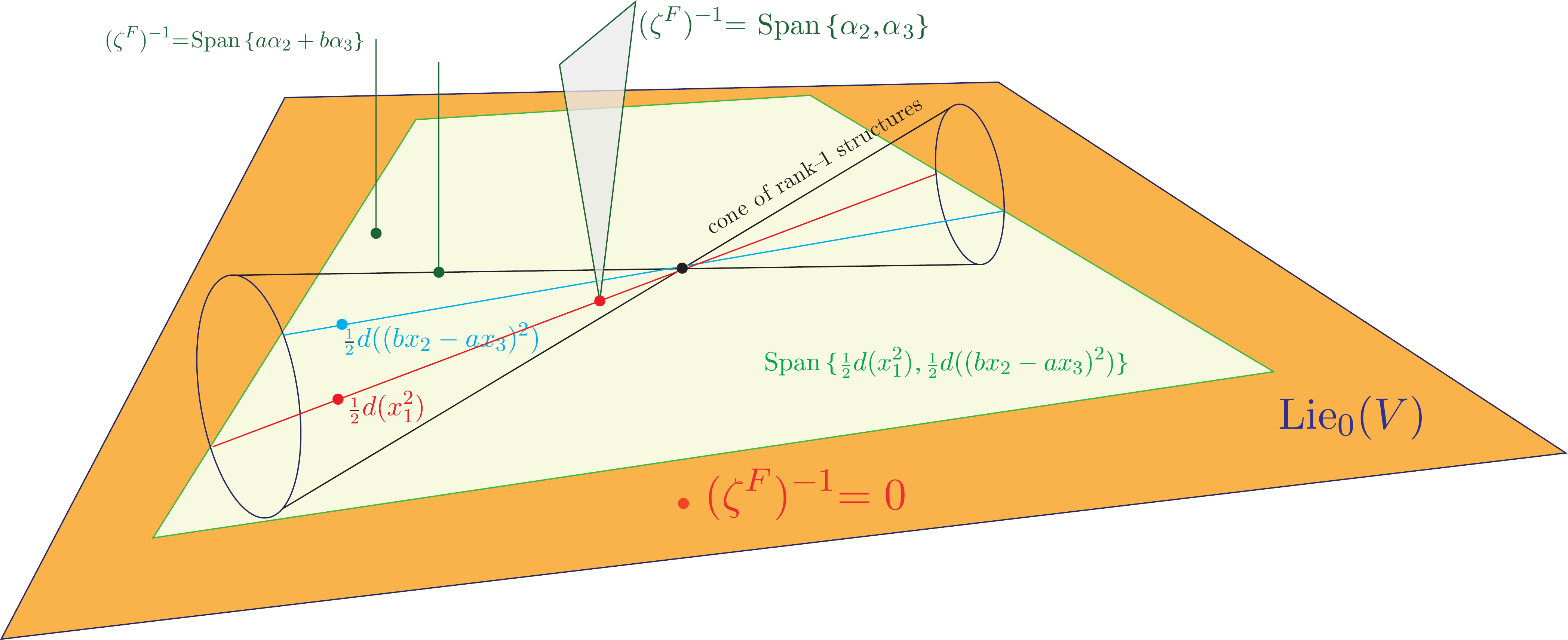,width=12cm}\caption{The compatibility variety af a  structure $c_F$ of type $A_1$.\label{figComVarRank1}}
\end{figure}

%Describe now $\frac{\Lie(V,c_F)}{\Stab(c_F)}$. The line $\Span{{\textstyle \frac{1}{2}}d(x_1^2)}$  is $\Stab(c_F)$--invariant, and on it the action is trivial. Then an orbit $\Omega$ in $\Span{{\textstyle \frac{1}{2}}d(x_1^2)}$ is just a point, and by the same arguments of Proposition \ref{propFibraRangoUno} it follows that $\frac{\zeta|^F_\Omega}{\sigma|^F_\Omega}$ has only one nonzero parallel section.\par
%On the other hand, $\Stab(c_F)$ causes the plane  $\Span{{\textstyle \frac{1}{2}}d(x_1^2), {\textstyle \frac{1}{2}}d((bx_2-ax_3)^2)}$ to rotate around the line 
%$\Span{{\textstyle \frac{1}{2}}d(x_1^2)}$. If $\Omega$ is the orbit of an element not belonging to the line $\Span{{\textstyle \frac{1}{2}}d(x_1^2)}$, then  the set of parallel sections of $\frac{\zeta|^F_\Omega}{\sigma|^F_\Omega}$ is $\F$.

\subsubsection{Compatibility varieties of $B_0$ structures}
If $\alpha_i$ (resp., $d(x_ix_j)$) is a  base vector  of $N$ (resp., $\Lie_0(V)$), then the dual to it covector will denoted by $\alpha_i^\circ$ (resp., $d(x_ix_j)^\circ$).\par
As it follows from Lemma \ref{lemLemmaCheRitenevoFineASeStessoMaChePoiAllaFineHoDovutoRichiamare}, the space of 2--cocycles of the structure $c_F+c_{\alpha_3}$, with $F={{\textstyle\frac{1}{2}} (\lambda x_1^2 +\mu x_2^2)}$, is the  6--dimensional space 
\begin{equation}\label{eqZetaDueNonUnimodulariBASE}
\Span{s^2,\alpha_1-\lambda d(x_1x_3),\alpha_2-\mu d(x_2x_3), \alpha_3}.
\end{equation} 
If $c$ is a  structure of type $B_0$, i.e., $\lambda=\mu=0$, then \begin{equation}\label{eqCompVarBZero}\Lie(V,c)=\zeta^{-1}(s^2).\end{equation}
$\zeta|_{\Lie(V,c)}$ is a stratified vector bundle over $s^2$. Indeed (see Lemma \ref{lemCNC}), $\zeta$ is of rank 3 over $\{0\}$, it is of rank 2 over the quadric $d(x_1^2)^\circ d(x_2^2)^\circ-(d(x_1x_2)^\circ)^2=0$, and it is of rank 1 over the rest of   $s^2$ (see Fig. \ref{figComVarNU0}).
\begin{figure}
\epsfig{file=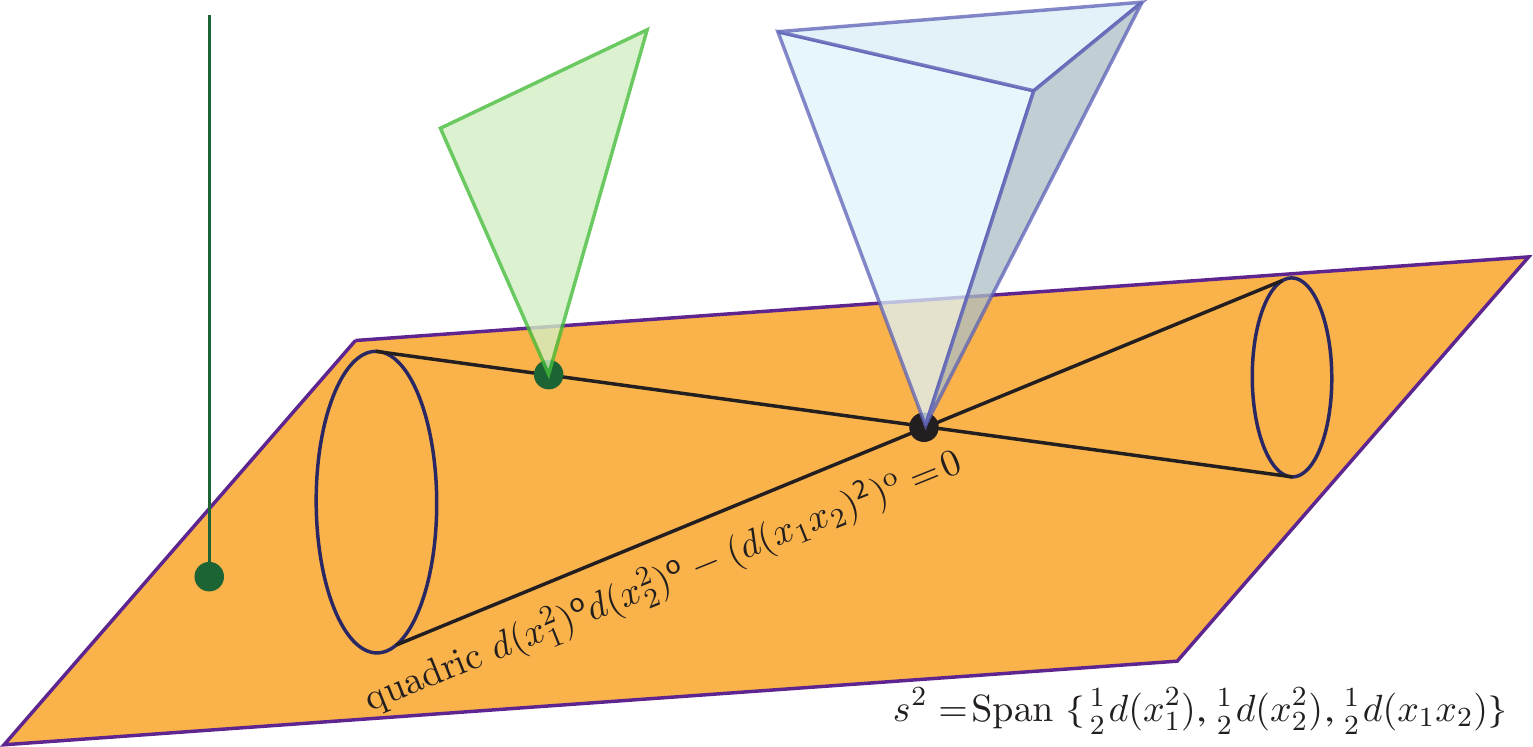,width=12cm}\caption{The compatibility variety of a non--unimodular Lie structure of type $B_0$.\label{figComVarNU0}}
\end{figure}
\subsubsection{Compatibility varieties of $B_1$ structures}

If $c$ is a  structure of type $B_1$, then $\lambda=1$ and $\mu=0$.  Directly from \REF{eqZetaDueNonUnimodulariBASE} it follows that $\Lie(V,c)$ is the intersection of $\zeta^{-1}(\Span{s^2, d(x_1x_3)})$ with the affine hyperplane   $(\alpha_1)^\circ=-(d(x_1x_3))^\circ$. Moreover, if $c_G+ad(x_1x_3)+\alpha\in\Lie(V,c)$, with   $c_G \in s^2$, it  is easy to prove that   $a=0$. In other words, \begin{equation}\label{eqCompVarBUno}\Lie(V,c)= \zeta^{-1}(s^2)\cap\{ (\alpha_1)^\circ =0\},\end{equation}
i.e., $\zeta|_{\Lie(V,c)}$ is a stratified vector bundle over $s^2$, whose fibers are subspaces of the corresponding  fibers of $\zeta$.\par Describe now the corresponding  strata.  
Let $c_G+\alpha\in\Lie(V,c)$. If $c_G\in\Span{{\textstyle \frac{1}{2}}d(x_1^2)}$ then $\zeta|_{\Lie(V,c)}^{-1}(c_G)=\Span{\alpha_2,\alpha_3}$. If $c_G$ is a point of the quadric $d(x_1^2)^\circ d(x_2^2)^\circ-(d(x_1x_2)^\circ)^2=0$, not belonging to the line $\Span{{\textstyle \frac{1}{2}}d(x_1^2)}$, then $\zeta|_{\Lie(V,c)}^{-1}(c_G)$ is the 1--dimensional subspace  $(\alpha_1)^\circ=0$   of $\zeta^{-1}(c_G)$. If $c_G$ is not in the quadric above,  then $\zeta|_{\Lie(V,c)}^{-1}(c_G)$ coincides with $\zeta^{-1}(c_G)$, i.e., $\Span{\alpha_3}$ (see Fig. \ref{figComVarNU1}).
\begin{figure}
\epsfig{file=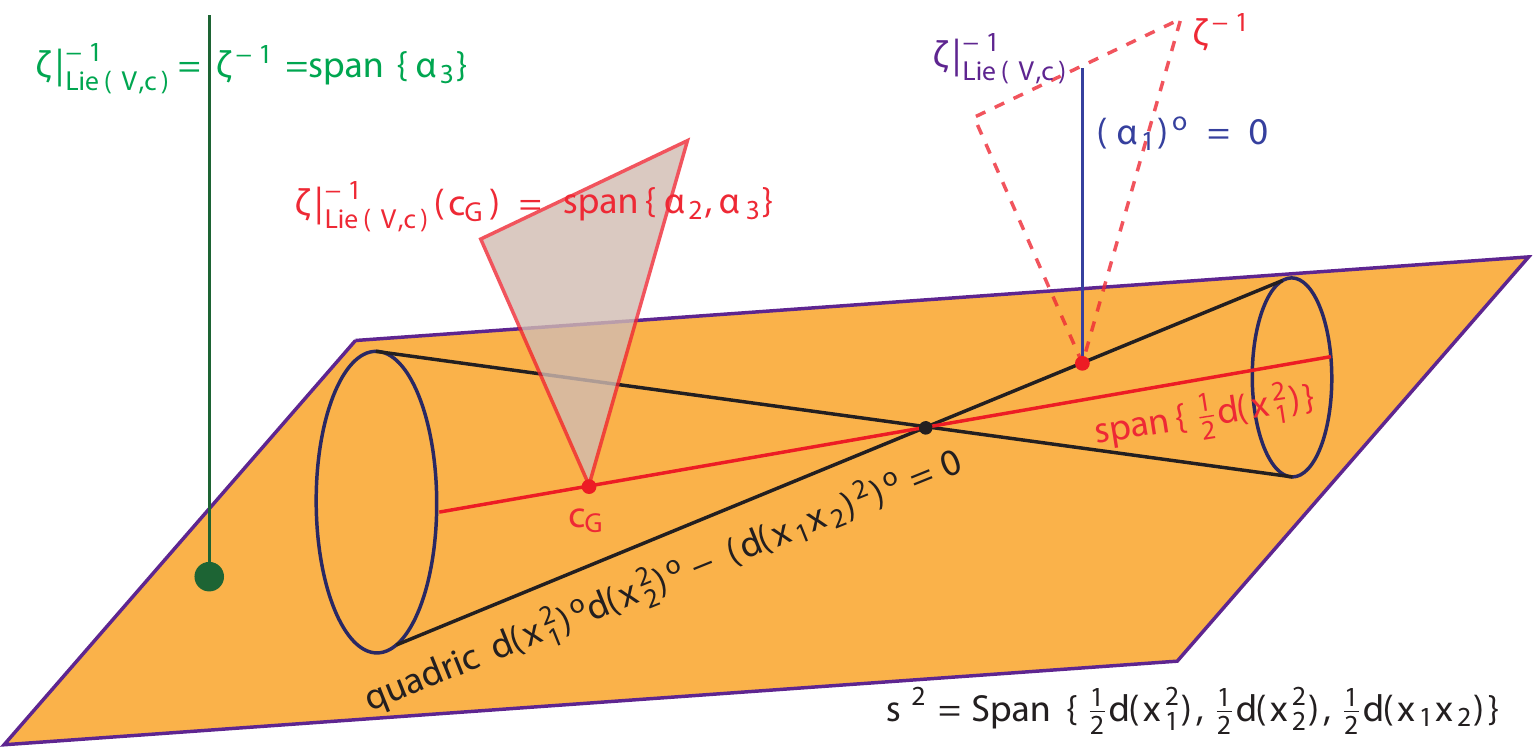,width=12cm}\caption{The compatibility variety of a non--unimodular Lie structure of type $B_1$.\label{figComVarNU1}}
\end{figure}

\subsubsection{Compatibility varieties of $B_{2,\nu}^\pm$ structures}

Finally, if $\lambda=\pm\mu=\nu^{-1}$, then $c$ is a structure of type $B_{2,\nu}^\pm$. In this case $\Lie(V,c)$ is the intersection  of $\zeta^{-1}(\Span{s^2, d(x_1x_3),d(x_2x_3)})$ with the   affine subspace 
$$
\left\{\begin{array}{l} (\alpha_1)^\circ=-\nu(d(x_1x_3))^\circ  \\  (\alpha_2)^\circ=\mp\nu(d(x_2x_3))^\circ .\end{array}\right.
$$
Moreover, if $c'=c_G+ed(x_1x_3)+fd(x_2x_3)+\alpha\in\Lie(V,c)$, with   $c_G \in s^2$, it  is easy to prove that   $e^2=\pm f^2$. \par
If $e=f=0$, i.e., the unimodular component of $c'$ belongs to $s^2$, then    \begin{equation}\label{eqCompVarBDue}\Lie(V,c)\cap \zeta^{-1}(s^2)= \zeta^{-1}(s^2)\cap\{ (\alpha_1)^\circ =(\alpha_2)^\circ=0\},\end{equation}
i.e., the restriction of $\zeta|_{\Lie(V,c)}$ over $s^2$ is a trivial vector bundle with the fiber $\Span{\alpha_3}$.\par
If $ef\neq 0$, then it is easy to prove that $c'=ac + e(d(x_1x_3)-\nu\alpha_1)+f(d(x_2x_3)\mp\nu\alpha_2)$. In other words, the restriction of $\zeta|_{\Lie(V,c)}$ over the degenerate quadric $\{(d(x_1x_3)^\circ)^2\mp (d(x_2x_3)^\circ)^2=0\}\subseteq\Span{{\scriptstyle \frac{1}{2}}d(x_1^2)\pm {\scriptstyle \frac{1}{2}}d(x_2^2), d(x_1x_3),d(x_2x_3)} $ is the graph of the map  
\begin{equation}\label{eqCompVarBDueAff}
 a({\scriptstyle \frac{1}{2}}d(x_1^2)\pm {\scriptstyle \frac{1}{2}}d(x_2^2))+ed(x_1x_3)+fd(x_2x_3)\longmapsto \nu(a\alpha_3-e \alpha_1\mp f \alpha_2).
\end{equation}
\par
Comparing \REF{eqCompVarBZero}, \REF{eqCompVarBUno}, \REF{eqCompVarBDue} and \REF{eqCompVarBDueAff}, one observes that when the rank of the unimodular component of $c$ increases, the dimension of the fibers of $\zeta|_{\Lie(V,c)}$ over $s^2$ decreases. Observe that in all cases,  $\Lie(V,c)\cap\Lie_0(V)=s^2$. It is worth also stressing that    elements $c'\in\Lie(V,c)$ such that $\zeta(c')\not\in s^2$  exists only for structures  $c$   of the type $B_{2,\nu}^\pm$ (see Fig. \ref{figComVarNU}).

\begin{figure}
\epsfig{file=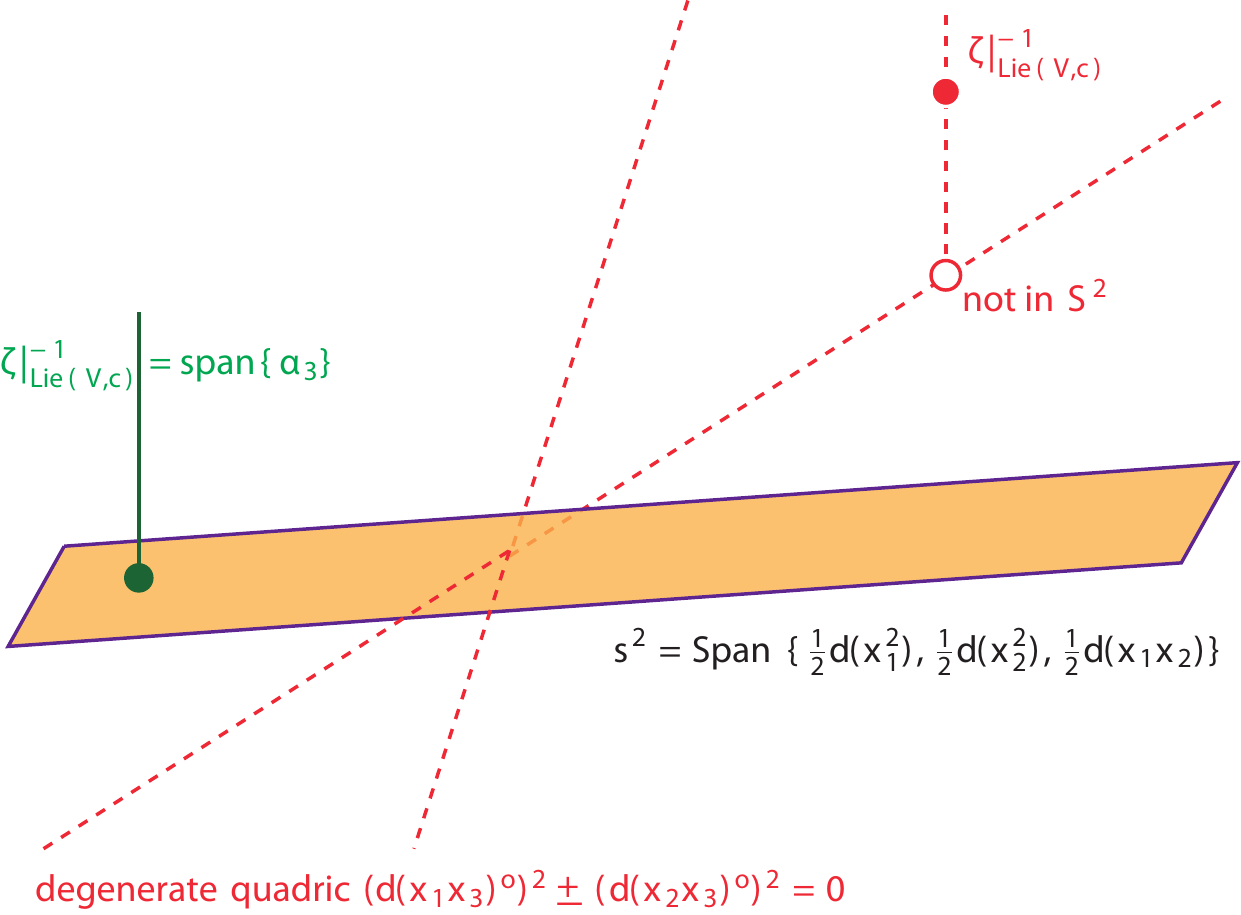,width=12cm}
\caption{The compatibility variety of a non--unimodular Lie structure of type $B_{2,\nu}^\pm$.\label{figComVarNU}}
\end{figure}

\subsection{Deformations of Lie structures}

Recall that  a (algebraic, smooth, continuous) deformation of a Lie structure $c$ is a (algebraic, smooth, continuous) curve in $\Lie(V)$, i.e., a map $\gamma:\F\longmapsto\Lie(V)$, passing through $c$. \par
%When required by the application (e.g., unitary representation in Q.M., simmetric solutions of E.E., etc.), the family $\{c_t\}$ must assumed to be smooth. Smoothness naturally leads to the concept of \emph{infinitesimal deformations}.
%, whose link with cohomology soon became evident [...],  and was later generalized to that of universal deformation [...].\par
%A particular interest was paid to those deformations of $c$ which, once projected on the   moduli space, reduce to two points, one of which is the orbit of $c$. They are called \emph{contractions} [...].\par
Denote by $\mathcal{F}$ the algebra of \emph{algebraic functions} on $\Lie(V)$, i.e.,   the quotient of $S(V\otimes_\F V)$ by the ideal generated by  \REF{eqAVrietˆBianchiInCoordinate}. If $\F=\R$, define also $C^\infty(\Lie(V))$ as the quotient of the algebra $C^\infty(V\otimes_\F V)$  by the ideal generated by  \REF{eqAVrietˆBianchiInCoordinate}. A map from $\F$ to $\Lie(V)$ is called alebraic (resp., \emph{smooth}) if it corresponds to an algebra homomorphism $\mathcal{F}\longmapsto \F[x]$ (resp., $C^\infty(\Lie(V))\longmapsto C^\infty(\R)$) in the sense of \cite{Jet}. In particular, a \emph{linear map} from $\F$ to $\Lie(V)$, i.e., an $\F$--homomorphsim from $\F$ to $V\otimes_\F V$ whose image is contained in $\Lie(V)$, is   algebraic (and smooth, if $\F=\R$). \par
A defomation is called \emph{linear} if   $\gamma$ is a straight line. Observe that the linear deformation 
\begin{equation}\label{eqDefLin}
\gamma_d(t)\df (1-t)c+td
\end{equation}
 of $c$ is naturally associated with the  element $d\in \Lie(V,c)$, $d\neq 0$.  
Obviously, any linear deformation of $c\in\Lie(V)$ is of the form $\gamma_d$.\par
We define an \emph{infinitesimal deformation} to be tangent vector at $c$ of a deformation $\gamma$. In particular, the infinitesimal deformation associated with $\gamma_d$ is the affine vector $\gamma'_d(0)$, connecting  $c$ and $d$. Infinitesimal deformations must be understood as elements of the tangent space to $\Lie(V)$.  Two infinitesimal deformations are called \emph{equivalent} if one is obtained from another by action of $d_c\psi$, with $\psi\in \Stab(c)$. \par
The tangent space  to $\Lie(V)$ is naturally identified with $Z^2(c)$, and the above described  action of $\Stab(c)$ coincides with a natural  action of $\Stab(c)$ on $Z^2(c)$. Moreover, the subset of $Z^2(c)$ that corresponds to the linear    deformations coincides with  $\Lie(V,c)$,  and the action of $\Stab(c)$ restricts to it.  
\subsection{Some examples of deformations}
Now we shall exploit the above description of   $\Lie(V,c)$  in order to describe deformations of a 3--dimensional Lie structure $c$  and   their equivalence classes as well. By abusing the language we shall call the quotient $\frac{\Lie(V,c)}{\Stab(c)}$ \virg{orbit space}.\par
To this end, it will be necessary to consider some special subgroups of $\GL(V)$.
%
%Accordingly to some authors, two deformations of $c$ are equivalent if there exists an element of $\Stab(c)$ which sends one deformation into the other.  Then   $\frac{\Lie(V,c)}{\Stab(c)}$ may be interpreted as the moduli space of linear deformations of $c$. \par
%
%The orbit space $\frac{\Lie(V,c)}{\Stab(c)}$ is  naturally interpreted as  the moduli space of linear deformations of $c$.\par
 %
%
\begin{remark}\label{remRemarkoGruppiOrtogonali}
If $F={ \frac{1}{2}}(x_1^2 +   x_2^2\pm  x_3^2)$, then $\Stab(c_F)$ is   $O(3)$ (resp., $O(2,1)$) (see also Proposition \ref{propSimmetrieSimmetriche}). Similarly, for $F={ \frac{1}{2}}(x_1^2  \pm  x_2^2)$,  the group $\Stab(c_F)$ will be denoted   $O(2,0)$ or $O(1,1,0)$, respectively.  
%A transformation which  leaves invariant the subspace  $\Span{ x_1, x_2}$   belongs to $O(2,0)$ (resp., $O(1,1,0)$) if, restricted on $\Span{x_1, x_2}$, is an element of  $O(2)$ (resp., $O(1,1)$). 
Finally, notice that $\Stab(c_F)$, for  $F={ \frac{1}{2}}x_1^2 $,  coincides with the stabilizer of $x_1$. We do not describe the orbits of the action of $\Stab(c_F)$ on $\Lie_0(V)$, since this concerns the theory of symmetric bilinear forms (see \cite{Lam}).
\end{remark}
Denote by   $p:\Lie(V)\longmapsto\frac{\Lie(V)}{\GL(V)}$   a natural projection of sets. 
%Let $c\in\Lie(V)$ be a Lie structure.
%In our framework we assume the following
Recall that  a (algebraic, smooth) deformation $\gamma$ of $c=\gamma(0)$ is called a \emph{contraction} of  $c$ if 
%\begin{definition}
%An algebraic (resp., smooth, in the case $\F=\R$) curve $\gamma:\F\longmapsto\Lie(V)$ such that $\gamma(0)=c$ is called a \emph{deformation} of $c$. A deformation is called \emph{linear} if such is $\gamma$. If 
$p\circ\gamma$ takes two different values for $t=0$ and $t\neq 0$. %, then $\gamma$ is called a \emph{contraction}.
%\end{definition}
%SPIEGARE $C^\infty(\Lie(V))$.\par
%A deformation is called a \emph{contraction} if $p\circ\gamma$ takes only two values, one of which is $\GL(V)\cdot c$, $p:\Lie(V)\longmapsto\frac{\Lie(V)}{\GL(V)}$ being the natural projection.\par

\subsubsection{Deformations of  $A_3^+$   structures}
Let $F={ \frac{1}{2}}(x_1^2 +   x_2^2 +  x_3^2)$ and $c=c_F$.
Then  $\Lie(V,c)=\Lie_0(V)$ (see Subsection \ref{subsubA3}), and $\Stab(c)=O(3)$ (see Remark \ref{remRemarkoGruppiOrtogonali}). Hence the orbit space  identifies with $\frac{S^2(V)}{O(3)}$, i.e., with  the space of diagonal 3 by 3 matrices over $\F$. \par
Observe that no deformation of $c$ is a contraction.
%\begin{remark}
%This example reveals that the above concept of equivalence of deformations is in some respects inadequate. Indeed, all small enough deformations of $c$ produce an equivalent to $c$ structure, and, as such, they should be considered \virg{trivial}, i.e., equivalent to the constant deformation $\gamma_c$.  Moreover, a deformation of a structure of type $A_3^-$ to  a structure of $A_3^+$, should not be considered at all, since the associated to it straight line will necessarily pass through a rank--2 structure. \par
%There should be given a different notion of deformation of a Lie structure which allows a structure to be deformed only within the closure of its orbit, and a different notion of equivalence of deformations, allowing all linear transformations.  We will not pursue this point of view.
% \end{remark}
 The reader should not confuse between deformations of Lie algebras and deformations of Lie algebra structures.

\subsubsection{Deformations of  $A_2^+$   structures}
Let $F={ \frac{1}{2}}(x_1^2 +   x_2^2)$ and $c=c_F$.\par
Observe that in this case $s^2$ is $O(2,0)$--invariant and the orbits of the restricted action of $O(2,0)$ are the same as the orbits of the natural action of $O(2)$ on $s^2$. It is easy to prove that the set of parallel sections of $\zeta^F|_\Omega$, for such an $\Omega$, is identified with $\F$.
\begin{remark} 
If intersection of two subspaces of a vector space is non--trivial, then there are smooth curves passing from one subspace to the other, in contrast with the algebraic ones. In particular there are smooth curves connecting any point of $\Lie_0(V)$ with any point of  $\Span{ s^2,\alpha_3}$ (See Figure \ref{figComVarRank2}). This is obviously not the case for algebraic curves. So, this example illustrates the difference between algebraic and smooth deformations.
\end{remark}
\subsubsection{Deformations of  $A_1$   structures}
Let $F={ \frac{1}{2}}(x_1^2 )$ and $c=c_F$.\par
In this case, the line $\Span{ \frac{1}{2}(x_1^2 )}$   is $\Stab(c)$--invariant and the  restricted action is trivial, i.e., $\Omega$ is a point. Similarly to Proposition \ref{propFibraRangoUno}, one proves that  there is only one nonzero parallel section of  $\zeta^F|_\Omega$.\par
Under the action of $\Stab(c)$, the plane $\Span{\frac{1}{2}(x_1^2 ),\frac{1}{2}((bx_2-ax_3)^2 )}$ (see Fig. \ref{figComVarRank1}) rotates around the axis  $\Span{ \frac{1}{2}(x_1^2 )}$. If  $\Omega$ is an orbit of $\Stab(c)$ not contained in this axis, then the set of parallel sections of $\zeta^F|_\Omega$ is identified with $\F$.

\subsection{Effect of deformations on symplectic foliation in the case $\F=\R$}
A deformation of a Lie structure $c$ induces a deformation of the   symplectic foliation of $P^c$. Note that only the solvable  3--dimensional Lie stuctures admit non--trivial deformations. 
% We consider now three examples of deformations, in the case $\F=\R$, and we describe the corresponding to them deformations of the symplectic foliation. It will be useful to review the symplectic foliation of a solvable Lie structure $c$.\par
In such a case, $P^c$ can be brought to the form
%the Poisson bi--vector associated with $c$ can be put in the form 
  \begin{equation}\label{eqStruttSol}
P_c=X_\phi\wedge\frac{\partial}{\partial x^3},
\end{equation}
   with $\phi\in\End(\R^2)$. Indeed, 
   solvable Lie structures $B_0$, $B_1$, $B_{2,\lambda}^\pm$, and  $A_2^\pm$ are of this form, with $\varphi$ being 
    $$\scriptsize{\left(\begin{array}{cc}-1 & 0 \\0 & -1\end{array}\right)},\ \scriptsize{\left(\begin{array}{cc}-1 & 1 \\0 & -1\end{array}\right)},\ \scriptsize{\left(\begin{array}{cc}-\lambda & 1 \\ \mp 1 & -\lambda\end{array}\right)},\textrm{{\normalsize\  and }}\scriptsize{\left(\begin{array}{cc}0 & 1 \\ \mp 1 & 0\end{array}\right)},$$ respectively. The nil-potent Lie structure $A_1$ corresponds to $\varphi=\scriptsize{\left(\begin{array}{cc}0 & 1 \\0 & 0\end{array}\right)}$.\par
%       
%   if $X_\phi=\phi_a^bx_b\frac{\partial}{\partial x_a}$, where $a,b=1,2$, then $P^c=\phi_a^bx_b\frac{\partial}{\partial x_a}\wedge\frac{\partial}{\partial x_3}$, and
%  \begin{equation}
%\alpha_c=\frac{1}{2}\left( \phi_2^1dx_1^2-\phi_1^2 dx_2^2\right)+\phi_2^2x_2dx_1-\phi_1^1x_1dx_2.
%\end{equation}
%Solvable Lie structures $B_0$, $B_1$, $B_{2,\lambda}^\pm$, and  $A_2^\pm$ are obtained when $\varphi$ is equal to $\scriptsize{\left(\begin{array}{cc}-1 & 0 \\0 & -1\end{array}\right)}$, $\scriptsize{\left(\begin{array}{cc}-1 & 1 \\0 & -1\end{array}\right)}$, $\scriptsize{\left(\begin{array}{cc}-\lambda & 1 \\ \mp 1 & -\lambda\end{array}\right)}$ and $\scriptsize{\left(\begin{array}{cc}0 & 1 \\ \mp 1 & 0\end{array}\right)}$, respectively. The nil-potent Lie structure $A_1$ is  determined by $\varphi=\scriptsize{\left(\begin{array}{cc}0 & 1 \\0 & 0\end{array}\right)}$, and the Abelian one, $A_0$,  by $\phi=0$ (see \cite{Marmo1994} for the theory of decomposable Poisson bi--vectors in higher dimensions). \par
If $X_\phi=\phi_a^bx_b\frac{\partial}{\partial x_a}$, $a,b=1,2$, then $P^c=\phi_a^bx_b\frac{\partial}{\partial x_a}\wedge\frac{\partial}{\partial x_3}$,
  \begin{equation}
\alpha_c=\frac{1}{2}\left( \phi_2^1dx_1^2-\phi_1^2 dx_2^2\right)+\phi_2^2x_2dx_1-\phi_1^1x_1dx_2,
\end{equation}
and, therefore, 
\begin{eqnarray*}
P^c_{x_1}&=&-\phi(x_1)\frac{\partial}{\partial x^3} ,\\
P^c_{x_2}&=&-\phi(x_2)\frac{\partial}{\partial x^3} ,\\
P^c_{x_3}&=&X_\phi.
\end{eqnarray*}
 Notice that in each point $p=(x_1,x_2,x_3)$ where $\phi(x_1)$ and  $\phi(x_2)$ are not simultaneously zero,  $\Span{ P^c_{x_1}, P^c_{x_2}}$ is the line generated by ${\scriptstyle\left.\frac{\partial}{\partial x^3}\right|_p} $. So, it holds the following lemma.
 \begin{lemma}
Symplectic leaves of a solvable  Lie structure corresponding to Poisson bi--vector \REF{eqStruttSol} are either pull--backs of   trajectories of $X_\phi$ in $\R^2\smallsetminus\ker\phi$ via the projection $\R^3\longmapsto\R^2$, or single points of the subspace $\ker\phi\oplus\langle x_3 \rangle$.
\end{lemma}

\subsubsection{Deformation of $B_{2,\lambda}^\pm$ to $A_2^\pm$}

In this case, elliptic (resp. hyperbolic) spirals converge to circles (resp. hyperbola), as $\lambda\to 0$. Since the $B_2^\lambda$'s are mutually non--isomorphic for different values of $\lambda$, such deformation is not a contraction.

\subsubsection{Deformation of $B_{2,1}^\pm$ to $A_1$}
Consider the family of structures $\{c_\mu^\pm\}_{\mu\in\R^+}$ of the form \REF{eqStruttSol}, with 
\begin{equation*}
\varphi_\mu^\pm={\left(\begin{array}{cc}-1 & 1 \\ \mp \mu & -1\end{array}\right)}
\end{equation*}
and  
\begin{equation*}
\alpha_{c_\mu^\pm}=\frac{1}{2}(dx_1^2\pm \mu dx_2^2)+\alpha_3.
\end{equation*}

Then the trajectory of $X_{\varphi_\mu^\pm}$ issuing from $(x_1^0,x_2^0)$, $x_1^0\neq 0$, is given by 

\begin{eqnarray*}
x_1(t)&=& e^{-t}\sqrt{(x_1^0)^2+\mu(x_2^0)^2}\cos\left(\arctan\left(\sqrt{\mu}\frac{x_2^0}{x_1^0}\right)+\sqrt{\mu}t\right)\\
x_2(t)&=& e^{-t}\sqrt{\frac{(x_1^0)^2}{\mu}+(x_2^0)^2}\sin\left(\arctan\left(\sqrt{\mu}\frac{x_2^0}{x_1^0}\right)+\sqrt{\mu}t\right)
\end{eqnarray*}

and

\begin{eqnarray*}
x_1(t)&=&  \frac{x_1^0+\sqrt{\mu}x_2^0}{2}e^{(\sqrt{\mu}- 1) t}+\frac{x_1^0-\sqrt{\mu}x_2^0}{2}e^{-(\sqrt{\mu}+1) t}  \\
x_2(t)&=& \frac{x_1^0+\sqrt{\mu}x_2^0}{2\sqrt{\mu}}e^{(\sqrt{\mu}-1) t}-\frac{x_1^0-\sqrt{\mu}x_2^0}{2\sqrt{\mu}}e^{-(\sqrt{\mu}+1) t}    
\end{eqnarray*}

Trajectories of $X_{\varphi_\mu^+}$ (red) and of $X_{\varphi_\mu^-}$ (blue), issuing from   vertices of a regular hexagon centered at the origin, are represented in  Figure \ref{DD}, for $\mu$ running from almost zero (first picture) to 1 (last picture). We see that both elliptic (determined by $c_\mu^+$) and hyperbolic (determined by $c_\mu^-$) spirals converge to the same foliation as $\mu\to 0$, and the constructed  deformation is a contraction.\par
%More precisely, 
%consider the family of structures $\{c_\mu^\pm\}_{\mu\in\R^+}$ of the form \REF{eqStruttSol}, where 
%\begin{equation*}
%\varphi_\mu^\pm={\left(\begin{array}{cc}-1 & 1 \\ \mp \mu & -1\end{array}\right)}
%\end{equation*}
%and consequently
%\begin{equation*}
%\alpha_{c_\mu^\pm}=\frac{1}{2}(dx_1^2\pm \mu dx_2^2)+\alpha_3.
%\end{equation*}

%Then the trajectory of $X_{\varphi_\mu^\pm}$ emanating from $(x_1^0,x_2^0)$, $x_1^0\neq 0$, is given by 

%\begin{eqnarray*}
%x_1(t)&=& e^{-t}\sqrt{(x_1^0)^2+\mu(x_2^0)^2}\cos\left(\arctan\left(\sqrt{\mu}\frac{x_2^0}{x_1^0}\right)+\sqrt{\mu}t\right)\\
%x_2(t)&=& e^{-t}\sqrt{\frac{(x_1^0)^2}{\mu}+(x_2^0)^2}\sin\left(\arctan\left(\sqrt{\mu}\frac{x_2^0}{x_1^0}\right)+\sqrt{\mu}t\right)
%\end{eqnarray*}

%and

%\begin{eqnarray*}
%x_1(t)&=&  \frac{x_1^0+\sqrt{\mu}x_2^0}{2}e^{(\sqrt{\mu}- 1) t}+\frac{x_1^0-\sqrt{\mu}x_2^0}{2}e^{-(\sqrt{\mu}+1) t}  \\
%x_2(t)&=& \frac{x_1^0+\sqrt{\mu}x_2^0}{2\sqrt{\mu}}e^{(\sqrt{\mu}-1) t}-\frac{x_1^0-\sqrt{\mu}x_2^0}{2\sqrt{\mu}}e^{-(\sqrt{\mu}+1) t}    
%\end{eqnarray*}

%In Figure \ref{DD} are portrayed the trajectories of $X_{\varphi_\mu^+}$ (red) and of $X_{\varphi_\mu^-}$ (blue), emanating from the vertexes of a regular hexagon centered at the origin, as $\mu$ ranges from nearly zero (first picture) to 1 (last picture).

\begin{figure}\caption{Projection on the $(x_1,x_2)$--plane of the symplectic leaves of the structures  $B_{2,1}^+$ (red) and $B_{2,1}^-$ (blue), as they undergo a simultaneous deformation   to $A_1$ (red and blue  overlapped).}\label{DD}
\epsfig{file=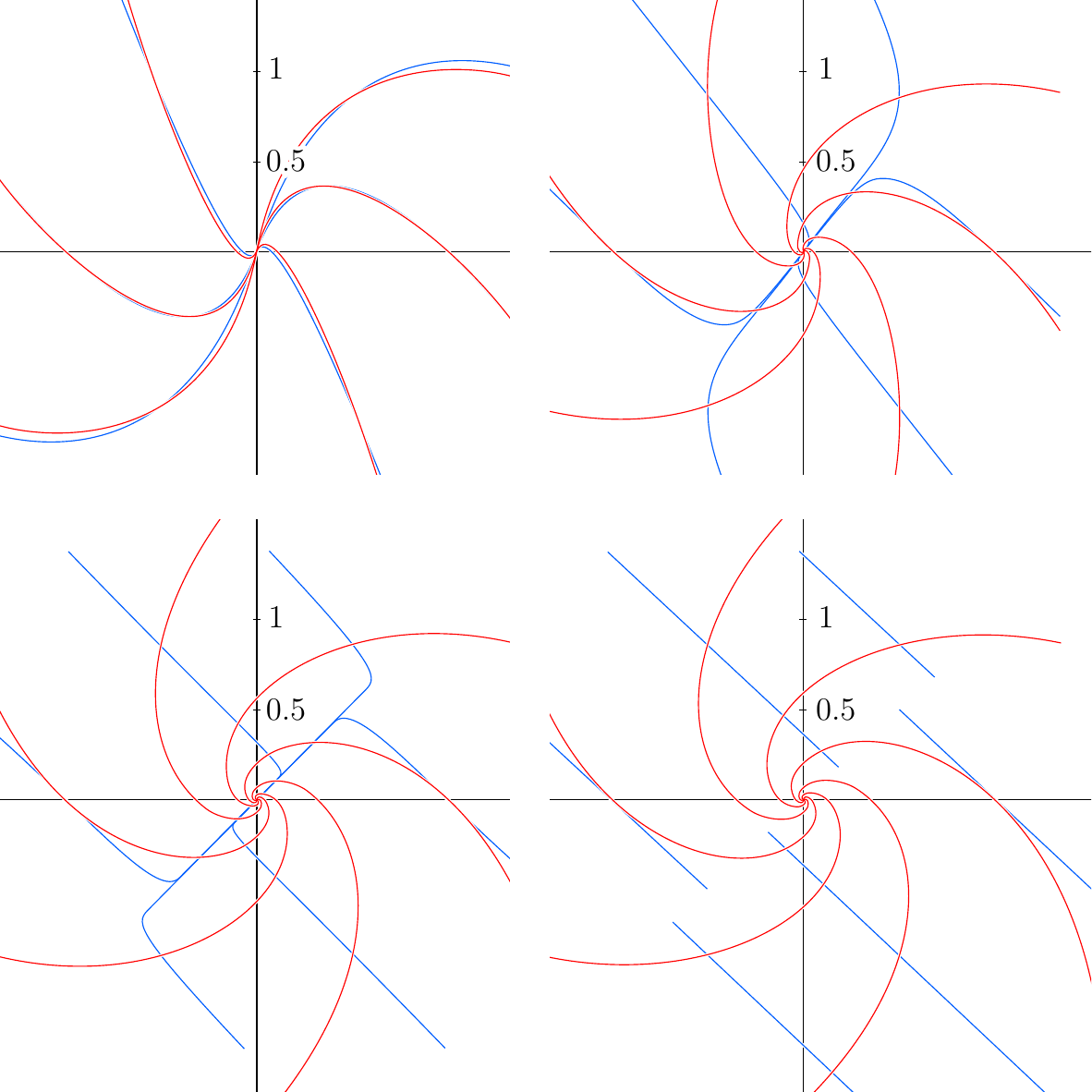, width=10cm}
\end{figure}

\subsubsection{Deformation of $B_1$ to $A_1$}

The deformation $\{c_\lambda\}_{\lambda\in\R}$ of the form \REF{eqStruttSol}, with 
\begin{equation*}
\varphi_\lambda={\left(\begin{array}{cc}-\lambda & 1 \\ \mp 0 &- \lambda\end{array}\right)}
\end{equation*}
and  
\begin{equation*}
\alpha_{c_\lambda}=\frac{1}{2}dx_1^2 + \lambda\alpha_3,
\end{equation*}
is a contraction. 
The trajectory of $X_{c_\lambda}$ issuing from $(x_1^0,x_2^0)$, which  is given by 
\begin{eqnarray*}
x_1(t)&=& x_1^0e^{-\lambda t} \\
x_2(t)&=& (x_1^0t+x_2^0)e^{-\lambda t} 
\end{eqnarray*}
and converges to the vertical straight line passing through $(x_1^0,x_2^0)$,  as  $\lambda\to 0$.

\subsubsection*{Aknowledgements}

The author is indebted to prof. Vinogradov, who carefully  supervised the works on the manuscript since its conception, and to prof. Marmo, for inspiring advices.

\end{document}